\documentclass[twocolumn]{autart}    

\usepackage{graphicx}          

\usepackage{amsmath}
\usepackage{amssymb}
\usepackage{bbold}
\usepackage{graphicx}   
\usepackage{subcaption}  

\usepackage{color}
\usepackage{appendix}
\usepackage{comment} 

\usepackage{algorithm}
\usepackage{algpseudocode}

\begin{document}

\begin{frontmatter}

\title{Positive stabilization of a pure diffusion system\thanksref{footnoteinfo}} 

\thanks[footnoteinfo]{Corresponding author V.~Piengeon. Tel. +3281724940.}

\author[UNamur]{Violaine Piengeon}\ead{violaine.piengeon@unamur.be},    
\author[UNamur]{Joseph J. Winkin}\ead{joseph.winkin@unamur.be}               
                                                                   
\address[UNamur]{University of Namur, Department of Mathematics \& naXys institute, Rue de Bruxelles 61, B-5000 Namur, Belgium} 


\begin{abstract}                          
This paper provides a complete analytical study of the positive stabilization, positive state estimation, and observer-based positive stabilization problems for a generic boundary control diffusion system with point observation. In particular, optimal design problems are solved analytically, and an original strategy for observer-based feedback stabilization is proposed. The three problems are also investigated for an arbitrary-order spatial discretization of the nominal PDE system enabling numerical implementation together with rigorous comparisons with the nominal solutions. In both the continuous and discretized settings, explicit and readily implementable solutions are derived.
\end{abstract}

\end{frontmatter}

\section{Introduction}
Positive linear systems constitute a class of dynamical systems whose state trajectories remain in the nonnegative cone of the state space whenever the initial condition and the input trajectory are nonnegative. This positivity property naturally arises when the state variables correspond to physical quantities such as concentrations, population densities, masses, or temperatures (in Kelvin), for which negative values have no physical interpretation. Because of this intrinsic constraint, positive linear systems have attracted considerable attention over the past decades and are present in numerous applications in biology, chemical and biochemical engineering, epidemiology, ecology, economics, or networked systems. The theory of positive finite-dimensional systems is now well established; see e.g. the books \cite{Farina2000}, \cite{Haddad2010}, \cite{Berman1989}, \cite{Berman1994}. Since many applications (diffusion processes, for example) are more accurately described by infinite-dimensional models, a theory of positive systems has also been developed specifically for infinite-dimensional systems; see e.g. \cite{Smith1995} for competitive and cooperative systems, \cite{Laabissi2003}, \cite{Laabissi2001} for chemical and biochemical systems, and \cite{Wintermayr2019} for systems with unbounded observation and control operators. The book  \cite{Batkai2017} includes connections with the positivity of finite-dimensional dynamical systems, via the concept of positive semigroup. The latter concept is thoroughly studied in the classical book \cite{Arendt1986}.

When studying stabilization and/or state estimation problems, positivity is an important and challenging additional constraint to take into account. Indeed, classical state feedback or observer design techniques (see e.g. \cite{Sontag2013}, \cite{Curtain_Zwart2020} or \cite{Yupanqui2021}) may produce controls or observers that violate positivity, thereby leading to physically meaningless results. Consequently, stabilization and state estimation methods must be specifically developed so that positivity is preserved while achieving the desired dynamical performance. More precisely, specifically for monitoring purposes, state observers should ideally be positive, just like closed-loop systems for stabilization.

The literature on stabilization and state estimation for positive linear finite-dimensional systems is well developed; see e.g. \cite{Rantzer2021} and the numerous references therein, or \cite{Rami2011} for observers and \cite{Beauthier2011} for optimal control. 
By contrast, for positive infinite-dimensional systems, that include diffusion systems, constructive methods for stabilization and/or state estimation remain comparatively limited, especially for the latter topic. In \cite{Abouzaid2010} and \cite{Achhab2017}, the authors study positive stabilization of linear infinite-dimensional systems with bounded control operators. The paper \cite{Anita2003} is dedicated to positive stabilization of a parabolic system by means of a geometric method. In \cite{Dehaye2016}, a detailed analysis of a standard example, namely the pure diffusion system, is performed in terms of positive stabilization. Locally positive stabilization of linear systems is investigated in \cite{Abouzaid2022} with illustrations on two classes of one-dimensional diffusion systems. However, in the references mentioned above, there is neither specific optimal control design, study of the state estimation problem, nor observer-based feedback design. Directly (and almost exclusively) related to the positive state estimation problem, \cite{Binid2021} provides a general framework but quite conservative conditions for designing positive observers.
Overall, to the best of our knowledge, no complete (in terms of optimality and implementability) analytical characterization is available in the literature for positively stabilizing exact state feedbacks, positive state observers and observer-based feedbacks within a unified framework. This lack of complete study even concerns diffusion models that are sufficiently simple to admit a complete analytical treatment.

The objective of this paper is to provide a comprehensive analytical study of the positive stabilization, positive state estimation, and observer-based positive stabilization problems for the one-dimensional boundary control diffusion system introduced in \cite{Dehaye2016} with point observation. Although deliberately simple, this model retains the essential features of boundary control diffusion systems with point observation while remaining analytically tractable. This tractability makes it possible to carry out a complete (in terms of optimality and implementability) analysis of the three fundamental positive design problems considered in this paper, leading to explicit analytical solutions. Beyond the specific model under investigation, the proposed analysis provides a unified framework for addressing positive stabilization and state estimation problems in boundary control diffusion systems and offers insight into methodologies that could be extended to more general classes of infinite-dimensional systems.

The main technical contributions of the paper are summarized as follows. First, building upon the characterization of positively stabilizing state-feedback laws established in \cite{Dehaye2016}, we formulate a relevant optimal control problem over the class of admissible feedbacks. We establish the corresponding feasibility conditions and derive the optimal solution in closed form. Second, thanks to connections with the former problem, we obtain the analogous results for positive state estimation. Third, we study observer-based positive stabilization. We prove that the considered system cannot be positively stabilized by an observer-based state feedback of the considered form and related to a positive state observer. This result is consistent with known limitations of the separation principle for positive state linear systems; see \cite{Binid2021} and \cite{Rami2011}. Motivated by this result, we propose an alternative design strategy. Rather than enforcing nonnegativity of the estimated state itself, the proposed approach guarantees nonnegativity of the real state trajectory for any nonnegative initial condition by identifying nontrivial and readily implementable admissible sets of initial estimation errors. Noticeably, all previous problems are adapted and investigated in parallel for a finite-dimensional spatial discretization of arbitrary order, introduced in \cite{Dehaye2016}. The corresponding problems are completely solved, analytically and in a readily implementable way, for an arbitrary discretization step, making the approach interesting in itself but also allowing rigorous comparisons between the solutions of the nominal and discretized models. 

The paper is organized as follows. The considered diffusion system and its discretized version are introduced in Section \ref{section:model}. Sections \ref{section:exact_feedback}, \ref{section:estimation} and \ref{section:observer_based_positive_stabilization} are dedicated to the positive stabilization, the positive state estimation, and the observer-based positive stabilization problems, respectively.
A conclusion with some perspectives is brought in Section \ref{section:conclusion}. Finally, some supplementary technical material is provided in the Appendix.

The following notations will be used throughout. The set $C(0,L)$ denotes the Banach space of continuous real-valued functions on $[0,L]$, while $\mathrm{H}^2(0,L)$ denotes the Sobolev space of functions belonging to the Lebesgue space $\mathrm{L}^2(0,L)$ whose first and second weak derivatives are also in $\mathrm{L}^2(0,L)$. $\mathbb{N}$ denotes the set of nonnegative integers. For any Banach lattice $B$, the set $B_+$ (resp. $B_-$) denotes the usual nonnegative (resp. nonpositive) cone. The relation $x\in B_+$ will also be denoted by $x\geq0$. The characteristic function of the set $S$ is denoted by $\mathbb{1}_S$. The first and last vectors of the canonical basis of ${\mathbb R}^n$ are denoted by $\mathbb{e}_1^{(n)}$ and $\mathbb{e}_n^{(n)}$, respectively, namely $ \mathbb{e}_1^{(n)} := ( \; \delta_{1i} \;  )_{i=1, \cdots,n}^\top$ and $ \mathbb{e}_n^{(n)} := ( \; \delta_{ni} \;  )_{i=1, \cdots,n}^\top$, where $\delta_{kl}$ is the usual Kronecker symbol. For any $j\in\mathbb{N}\setminus\{0\}$ and $l\in\{0,...,j\}$, the $l$-th elementary symmetric polynomial in $j$ variables $x_1,...,x_j$ is denoted by $\sigma_l([x_i]_{i=1}^j):=\sigma_l(x_1,...,x_j)$, where $\sigma_0(x_1,...,x_j)=1$ and $\sigma_l(x_1,...,x_j)=\sum_{1\leq i_1<...<i_l\leq j}x_{i_1}x_{i_2}\cdots x_{i_l}$ otherwise; see \cite[Example 3.84]{Vinberg2003}. For any linear operator $A$, the set $\sigma(A)$ denotes its spectrum. For a square matrix $M \in {\mathbb R}^{n\times n}$ and for any integer $j$ such that $1\leq j \leq n$, the symbol $[M]_{|j}$ ($[M]_{|j|}$, respectively) denotes the principal minor of $M$ that is obtained by computing the determinant of the $j\times j$-submatrix of $M$ constituted by the $j$ last  rows and columns of $M$ (by $j$ successive rows and the corresponding columns of $M$, the first and last rows and columns being excluded, respectively). 

\section{Model description}\label{section:model}
\subsection{Diffusion system}
Let us consider the boundary controlled diffusion system described by the partial differential equation (PDE) 
\begin{equation}\label{eq:pure_diffusion}
    \frac{\partial x}{\partial t}(z,t) = D_a\frac{\partial^2 x}{\partial {z}^2}(z,t), \quad z\in [0,L] ,\: t\geq 0  ,
\end{equation}
with Neumann boundary conditions 
\begin{equation}\label{eq:cdb_control}
    \dfrac{\partial x}{\partial z}(0,t)+u(t) = 0,\quad
    \dfrac{\partial x}{\partial z}(L,t) = 0,
\end{equation}
where~$D_a$ and $L$ are positive constants, $x(\cdot,t)$ is the state at time $t$ and $u$ denotes the boundary control. The measured output $y$ that is considered consists in a boundary observation given by
\begin{equation}\label{eq:output}
    y(t) = x(L,t), \quad t\geq 0.
\end{equation}
The PDE~\eqref{eq:pure_diffusion} is also called "heat equation" in the literature. Indeed, it can be seen as a description of the phenomenon of heat conduction in a metal rod of length~$L$; see e.g. \cite[Subsection 6.1.3]{VandeWouwer2014} or \cite[Section 1]{Crank1975}. In this setting,~$t$ is a temporal variable, and~$z$ is a spatial variable corresponding to the position along the rod. The parameters~$D_a$ and~$L$ denote the diffusion parameter and the domain length, respectively. The quantity~$x(z,t)$ represents the temperature of the rod at point~$z$ and time~$t$, and an initial condition~$x(z,0) = x_0(z), \;\; z\in [0,L]$, is referred to as the initial temperature profile. From a physical point of view, the boundary conditions~\eqref{eq:cdb_control} correspond, for example, to a metal rod that is heated or cooled only at one end and insulated at the other end, where a sensor is placed; see~\eqref{eq:output}.

Note that the uncontrolled (i.e. $u=0$) diffusion system is positive as the associated operator defined by $A_{0} = D_a({\rm  d}^2/{\rm  d}  z^2)$ with domain $    D(A_0) = \{x\in \mathrm{H}^2(0,L) : \frac{{\rm  d}  x}{{\rm  d}  z}(0) =  0 = \frac{{\rm  d}  x}{{\rm  d}  z}(L)  \}$
generates a positive $C_0$-semigroup on the state space $\mathrm{L}^2(0,L)$; see \cite{Laabissi2001}. However, the system is not exponentially stable since the spectrum of $A_0$ contains zero; see \cite[Example 3.2.15]{Curtain_Zwart2020}. 

\subsection{Approximate model}
Here, we present the approximate model which will be used throughout. This approximation is obtained by discretizing the system \eqref{eq:pure_diffusion}-\eqref{eq:output} with respect to the spatial variable $z$. To this end, we consider $n$ discretization points $z_i$, $i=1,\dots,n$, with $z_1=0$ and $z_n=L$. The discretization step, denoted by $\Delta z$, is given by~$\Delta z = L/(n-1)$.
Using a finite difference scheme as in \cite[Section 3.2]{Dehaye2016}, we obtain the discretized diffusion system 
\begin{equation}\label{eq:pure_diffusion_discr_control}
    \dot{x}^{(n)}(t) = A^{(n)}x^{(n)}(t)+b^{(n)} u(t),\quad t\geq 0,
\end{equation}
where 
\begin{equation*}\label{eq:matrice_An}
    A^{(n)} = 
    \begin{pmatrix}
        -p_2 & p_2 & 0 & \cdots  & 0\\
        p_2 & -2p_2 & p_2 &  & \vdots\\
        0 & \ddots & \ddots & \ddots  & 0\\
        \vdots & & p_2 & -2p_2 & p_2\\
        0  & \cdots & 0 & p_2 & -p_2
    \end{pmatrix}\in {\mathbb R}^{n\times n},
\end{equation*}
is a tridiagonal matrix, 
and
\begin{equation}\label{eq:p1 p2}
p_1 = D_a(\Delta z)^{-1},\quad p_2 = D_a(\Delta z)^{-2}.
\end{equation}
Any state trajectory of~\eqref{eq:pure_diffusion_discr_control}
is defined by
\begin{equation*}\label{eq:vecteur_xn}
    x^{(n)}(\cdot) :=
    \begin{pmatrix}
        x(z_1,\cdot) & \cdots & x(z_n,\cdot)   
    \end{pmatrix}^{\top} \in {\mathbb R}^n.
\end{equation*}
Moreover, the discretized output is given by 
\begin{equation}\label{eq:output_discr}
    y(t) = {c^{(n)}}^{\top}x^{(n)}(t), \quad t\geq 0,
\end{equation}
where $c^{(n)} = \mathbb{e}_n^{(n)}$.

One of the advantages of this model is that it shares common features with
the nominal system. Indeed, as mentioned in \cite[Section 3.2]{Dehaye2016}, the discretized diffusion system \eqref{eq:pure_diffusion_discr_control}-\eqref{eq:output_discr} is a positive LTI system, that is, $A^{(n)}$ is a Metzler matrix and $b^{(n)}$ and $c^{(n)}$ are nonnegative vectors. Moreover, \eqref{eq:pure_diffusion_discr_control} is not exponentially stable with zero as the only eigenvalue in the unstable spectrum of $A^{(n)}$. In addition to being tridiagonal, the matrix $A^{(n)}$ is symmetric, and so all its eigenvalues are real.

\section{Positively stabilizing state feedbacks}\label{section:exact_feedback}
As the diffusion system is not exponentially stable, it makes sense to try designing a state feedback law in such a way that the resulting closed-loop system is exponentially stable. To ensure the consistency of the model, this state feedback has to be chosen by paying attention to maintain the state nonnegativity property. It turns out that an exact state feedback law of the form
\begin{equation}\label{eq:state_feedback_pure_diff}
    u(t) = \kappa x(0,t), \quad t\geq 0,
\end{equation}
where $\kappa\in\mathbb{R}$, is positively exponentially stabilizing for the diffusion system,  
i.e. the operator 
$A_{\kappa} = D_a({\rm  d}^2/{\rm  d}  z^2)$, with domain
\begin{equation*}\label{eq:operator_Ak}
   D(A_{\kappa}) = \left\{x\in \mathrm{H}^2(0,L) : \dfrac{{\rm  d}  x}{{\rm  d}  z}(0) + \kappa x(0) = 0 = \dfrac{{\rm  d}  x}{{\rm  d}  z}(L) \right\}
\end{equation*}
is the infinitesimal generator of a positive and exponentially stable $C_0$-semigroup $(T_{A_{\kappa}}(t))_{t\geq0}$ on $\mathrm{L}^2(0,L)$ if, and only if, $\kappa < 0$, or equivalently, its growth bound, denoted by $\omega_0^{\kappa}$, is negative.
Indeed, $(T_{A_{\kappa}}(t))_{t\geq0}$ is positive for all $\kappa\in\mathbb{R}$. See \cite[Theorem 4]{Dehaye2016}. Similar results may also be derived for the discretized system. In that context, the state feedback corresponding to \eqref{eq:state_feedback_pure_diff} is given by 
\begin{equation}\label{eq:state_feedback_pure_diff_discr}
    u(t) =
    \kappa {\mathbb{e}_1^{(n)}}^{\top}x^{(n)}(t),\quad t\geq 0.
\end{equation}
\begin{rem}\label{rem:A_kappa_property}
    For $n$ sufficiently large, the closed-loop discretized system obtained by plugging \eqref{eq:state_feedback_pure_diff_discr} in \eqref{eq:pure_diffusion_discr_control} is positive and exponentially stable, that is the matrix
    \begin{equation*}\label{eq:matrix_Ak1}
        A^{(n)}_{\kappa} := A^{(n)}+\kappa p_1 \mathbb{e}_1^{(n)}{\mathbb{e}_1^{(n)}}^{\top}
    \end{equation*}
    is Metzler and stable, if, and only if,~$\kappa < 0$. Observe that $A^{(n)}_{\kappa}$ is Metzler for all $\kappa\in\mathbb{R}$. See \cite{Dehaye2016}.
\end{rem}

Since there exists an infinite number of state feedbacks of the form \eqref{eq:state_feedback_pure_diff} that are positively exponentially stabilizing for the diffusion system, the goal of this section is to find the best one according to a well-chosen optimality criterion and to develop a numerical scheme for its effective implementation on a discretized model. For any $\alpha>0$, the following optimal control problem $\mathbf{S_{\alpha}}$ is considered for the diffusion system: 
\begin{subequations}\label{eq:pbm_opti}
    \begin{align}
        \displaystyle \min_{\kappa \,\in \,\mathbb{R}}\quad &\kappa^2, \label{eq:pbm_opti_obj}\\
        \text{s.t.}\quad & (T_{A_{\kappa}}(t))_{t\geq0} \text{ is positive,} \label{eq:pbm_opti_const_positivity}\\
        \quad & (T_{A_{\kappa}}(t))_{t\geq0} \text{ is } \beta\text{-exponentially stable},\notag\\ 
        \quad & \,\,\forall \beta\in(-\alpha,0].\label{eq:pbm_opti_const_stability}
    \end{align}
\end{subequations}
\begin{rem}
$\alpha$) Condition~\eqref{eq:pbm_opti_const_stability} is a stability constraint which is stronger than exponential stability: it is equivalent to $\omega_0^{\kappa}\leq-\alpha$; see \cite[Section 4.1]{Curtain_Zwart2020}. The latter inequality ensures the feasibility of $\mathbf{S_{\alpha}}$. 
 In condition \eqref{eq:pbm_opti_const_stability}, the parameter $\alpha$ corresponds to the required minimum stability margin for $(T_{A_{\kappa}}(t))_{t\geq0}$. \\
$\beta$) 
A state feedback of the form~\eqref{eq:state_feedback_pure_diff} where $\kappa$ satisfies~\eqref{eq:pbm_opti_const_positivity} and~\eqref{eq:pbm_opti_const_stability} may be viewed as a constrained low gain (static) state feedback. 
The low gain interpretation comes from the fact that the input~\eqref{eq:state_feedback_pure_diff} is parametrized by the scalar $\kappa$ and converges to zero as $\kappa$ tends to zero. In this context, the choice of the optimality criterion~\eqref{eq:pbm_opti_obj} is motivated  by key features of low gain type state feedbacks such as avoiding actuator saturation by decreasing the value of the low gain parameter; see \cite{Lin2009LowGA}.\\
$\gamma$) Consider the closed-loop transfer function ${\hat g}_{cl}$ (of system \eqref{eq:pure_diffusion}-\eqref{eq:cdb_control} under the feedback \eqref{eq:state_feedback_pure_diff}) from an external input $v(t)$ acting at $z=0$, such that $u(t) = \kappa x(0,t)+v(t)$, to the output  $y(t)=x(L,t)$. This transfer function is given by ${\hat g}_{cl}(s) = \left( \sqrt{s} \sinh{\sqrt{s} - \kappa \cosh{\sqrt{s}}}\right)^{-1}$. Therefore, $\mathbf{S_{\alpha}}$ can be interpreted as maximizing the static gain, i.e. the asymptotic value of the step response of the closed-loop system, under the constraints \eqref{eq:pbm_opti_const_positivity}--\eqref{eq:pbm_opti_const_stability}.
\end{rem}
In the light of those comments, $\mathbf{S_{\alpha}}$ may be viewed as a trade-off 
between a sufficiently high stability margin and the advantages offered by a sufficiently low gain state feedback, while maintaining positivity of the state trajectories.
With a view to solving $\mathbf{S_{\alpha}}$, for any $\alpha>0$, numerically, the analysis is based on the optimal control problem $\mathbf{S_{\alpha}^{(n)}}$ for the discretized diffusion system:
\begin{subequations}\label{eq:pbm_opti_discr}
    \begin{align}
        \displaystyle \min_{\kappa \,\in \,\mathbb{R}}\quad &{\kappa}^2, \label{eq:pbm_opti_discr_obj}\\
        \text{s.t.}\quad & A^{(n)}_{\kappa} \text{ is Metzler,} \label{eq:pbm_opti_discr_const_positivity}\\
        \quad & A^{(n)}_{\kappa} \text{ is } \beta\text{-exponentially stable}, \forall \beta\in(-\alpha,0]\,.\label{eq:pbm_opti_discr_const_stability}
    \end{align}
\end{subequations} 
Clearly, \eqref{eq:pbm_opti_discr_const_stability} is  equivalent to the condition
$$\lambda_d^{\kappa} := \max \{\lambda : \lambda\in\sigma(A^{(n)}_{\kappa})\} \leq-\alpha \; .$$
As for the $\mathbf{S_{\alpha}}$ problem, the feasible set of the $\mathbf{S_{\alpha}^{(n)}}$ problem is the set of all values of $\kappa$ that yield a positively exponentially stabilizing state feedback~\eqref{eq:state_feedback_pure_diff_discr} with an associated stability margin greater or equal to $\alpha$.

\begin{lem}\label{lemma:feasability_practicality}
For all ~$\alpha>0$, the $\mathbf{S_{\alpha}^{(n)}}$ problem~\eqref{eq:pbm_opti_discr} for the discretized diffusion system \eqref{eq:pure_diffusion_discr_control} is feasible, i.e. there exists $\kappa$ such that \eqref{eq:pbm_opti_discr_const_positivity} and~\eqref{eq:pbm_opti_discr_const_stability} hold, if, and only if,~$-\alpha > \alpha_f^{(n)}$, where 
\begin{equation*}
    \alpha_f^{(n)} := \max \{\lambda \in \mathbb{R}: [A^{(n)}-\lambda I]_{|n-1} = 0 \}.
\end{equation*}
In that case, for all $\kappa\in\mathbb{R}$, conditions~\eqref{eq:pbm_opti_discr_const_positivity} and~\eqref{eq:pbm_opti_discr_const_stability} hold if, and only if,
\begin{equation}\label{eq:designk1}
    \kappa\leq\frac{p_2-\alpha}{p_1}+\frac{p_2^2}{p_1}\frac{[A^{(n)}+\alpha I]_{|n-2}}{[A^{(n)}+\alpha I]_{|n-1}}.
\end{equation}
\end{lem}
\begin{pf}
First, recall that $A^{(n)}_{\kappa}$ is a Metzler matrix for all $\kappa\in\mathbb{R}$. So we focus on condition~\eqref{eq:pbm_opti_discr_const_stability}. Let~$\kappa\in\mathbb{R}$. A direct computation using cofactor expansion along the first column of $A_{\kappa}^{(n)} - \lambda I$, where $\lambda\in\mathbb{R},$ yields the fact that the (real) eigenvalues of $A^{(n)}_{\kappa}$ correspond to the zeros of the function $g_{\kappa}$, where, for any $\eta \in {\mathbb R}$, the function $g_{\eta} : {\mathbb R}\setminus\{\lambda \in {\mathbb R} : [A^{(n)}-\lambda I]_{|n-1} = 0 \}\rightarrow \mathbb{R}$ is defined by
    \begin{equation*}
        g_{\eta}(\lambda) := \eta -\frac{p_2+\lambda}{p_1}-\frac{p_2^2}{p_1}\frac{[A^{(n)}-\lambda I]_{|n-2}}{[A^{(n)}-\lambda I]_{|n-1}}
    \end{equation*}
(see Appendix A).   
Note that $\alpha_f^{(n)}$ is negative since the matrix generated from $A^{(n)}$ by removing its first column and row, is negative definite. Clearly, the function $g_{\kappa}$ is rational and therefore continuous on its domain. Moreover it can be shown that $g_{\kappa}$ is strictly decreasing on $(\alpha_f^{(n)},+\infty)$, with $\lim_{\lambda \searrow \alpha_f^{(n)}} g_{\kappa}(\lambda) = +\infty$ and $\lim_{\lambda \to +\infty} g_{\kappa}(\lambda) = -\infty$.
Thus, $g_{\kappa}$ admits a unique zero $\lambda^* > \alpha_f^{(n)}$, such that $\lambda^* = \lambda_d^{\kappa} > \alpha_f^{(n)}$. Hence, $\kappa$ being arbitrary, a necessary condition for \eqref{eq:pbm_opti_discr_const_stability} to hold is $-\alpha >\alpha_f^{(n)}$. Conversely, under this condition, as $g_{\kappa}$ is (strictly) decreasing on $(\alpha_f^{(n)},+\infty)$, it is clear that $\lambda_d^{\kappa} \leq-\alpha$ holds if, and only if, $g_{\kappa}(-\alpha) \leq g_{\kappa}(\lambda_d^{\kappa}) = 0$, which yields inequality \eqref{eq:designk1}. It follows that the condition $-\alpha >\alpha_f^{(n)}$ is also sufficient for $\mathbf{S_{\alpha}^{(n)}}$ to be feasible (by choosing $\kappa$ such that \eqref{eq:designk1} is satisfied). \qed
\end{pf}
\begin{rem}
The feasibility condition on $\alpha$ in Lemma \ref{lemma:feasability_practicality}
means that there are no values of $\kappa$ that give a stability margin related to $A^{(n)}_{\kappa}$ that is greater or equal to $-\alpha_f^{(n)}$: it is not possible to ensure any (i.e. too large) stability margin for the closed-loop system \eqref{eq:pure_diffusion_discr_control}, \eqref{eq:state_feedback_pure_diff_discr} while maintaining its positivity. 
\end{rem}
Now we are able to prove the following result which provides, for any $n$, an analytical and explicit expression of the solution of the $\mathbf{S_{\alpha}^{(n)}}$ problem when it is feasible. 
\begin{thm}\label{theorem:design_k1_opti}
Let~$\alpha\in\mathbb{R}$ be such that $\; \alpha_f^{(n)} <  -\alpha < 0$. 
The $\mathbf{S_{\alpha}^{(n)}}$ problem~\eqref{eq:pbm_opti_discr} for the discretized diffusion system \eqref{eq:pure_diffusion_discr_control} has a unique solution~$\kappa_*^{(n)}$, which is given by
    \begin{equation}\label{eq:solution_discr}
        \kappa_*^{(n)} =\frac{1}{\sqrt{D_a}} \cdot
        \frac{\displaystyle\sum_{l=0}^{n-1} \binom{n+l}{2l+1}\frac{(-\alpha)^{l+1}}{p_2^{l+\frac{1}{2}}}}{\displaystyle\sum_{l=0}^{n-1} \binom{n+l-1}{2l}\frac{(-\alpha)^{l}} {p_2^{l}}}.
    \end{equation}
\end{thm}
\begin{pf} 
Since $-\alpha >\alpha_f^{(n)}$, Lemma~\ref{lemma:feasability_practicality} ensures that $\mathbf{S_{\alpha}^{(n)}}$ is feasible. Moreover, the function $g_{0}$ is strictly decreasing on $(\alpha_f^{(n)}, +\infty)$, 
therefore $g_0(-\alpha)>g_0(0)$. It can be shown by induction that $[A^{(n)}]_{|n-2} / [A^{(n)}]_{|n-1} = -1/p_2$ (see Appendix B), which implies that $g_0(0)=0$. Consequently, the right-hand side of  \eqref{eq:designk1}, i.e. $-g_0(-\alpha)$, is negative. Thus, in view of Lemma \ref{lemma:feasability_practicality}, the $\mathbf{S_{\alpha}^{(n)}}$ problem \eqref{eq:pbm_opti_discr} has a unique solution~$\kappa_*^{(n)}$, which is given by 
\begin{equation}\label{eq:solution_discr_exp1}
    \kappa_*^{(n)} = \frac{p_2-\alpha}{p_1}+\frac{p_2^2}{p_1}\frac{[A^{(n)}+\alpha I]_{|n-2}}{[A^{(n)}+\alpha I]_{|n-1}}.
\end{equation}
To complete the proof, it remains to derive the closed-form expression \eqref{eq:solution_discr} from \eqref{eq:solution_discr_exp1}. Using a cofactor expansion along the last column yields, for any $i=3,\dots,n-1$, the identity
\begin{equation}\label{eq:det_cofactor_exp}
        [A^{(n)}+\alpha I]_{|i} = (-p_2+\alpha)[A^{(n)}+\alpha I]_{|i-1|}-p_2^2[A^{(n)}+\alpha I]_{|i-2|}\,.
\end{equation}
For all $j=1,\dots,n$,  $[A^{(n)}+\alpha I]_{|j|}$ is the determinant of a symmetric tridiagonal Toeplitz matrix whose $j$ eigenvalues are given by $-2p_2 + \alpha + 2p_2\cos(l\pi/(j+1))$, $l = 1,\dots,j$; see \cite[Section 2]{Noschese2013}. By expanding this determinant and using the Viète's formulas (see \cite[Remark 3.14]{Vinberg2003}), we obtain
\begin{align}
[A^{(n)}+\alpha I]_{|j|} 
= \displaystyle\prod_{l=1}^j \left(-2p_2+\alpha+2p_2\cos\left(\frac{l\pi}{j+1}\right)\right)&\notag \\
= \displaystyle\sum_{l=0}^j \sigma_l\left(\left[-2 +2\cos\left(\frac{i\pi}{j+1}\right)\right]_{i=1}^j\right) p_2^{l}\alpha^{j-l},\label{eq:det_poly_sym_elem}&
\end{align}
where
\begin{equation*}
    \sigma_l\left(\left[-2+2\cos\left(\frac{i\pi}{j+1}\right)\right]_{i=1}^j\right) = \binom{2j-l+1}{l}(-1)^l\,.
\end{equation*}
Plugging \eqref{eq:det_poly_sym_elem} for $j=i-1$ and $j=i-2$ into \eqref{eq:det_cofactor_exp} and simplifying the resulting expression by means of binomial identities leads to
\begin{equation}\label{eq:det_Aa_i}
[A^{(n)}+\alpha I]_{|i} = (-1)^i p_2^i \sum_{l=0}^{i} \binom{i + l}{2l} \frac{(-\alpha)^l}{p_2^l}.
\end{equation}
Rewriting the right-hand side of \eqref{eq:solution_discr_exp1} as a quotient and applying the latter formula for $i = n-1$ and $i = n-2$  yields the following expression for the numerator of \eqref{eq:solution_discr_exp1}: 
\begin{align*}
(p_2 - \alpha)[A^{(n)}+\alpha I]_{|n-1} + p_2^2[A^{(n)}+\alpha I]_{|n-2} =&\\
(-1)^{n-1} p_2^{n-1} \sum_{l=0}^{n-1} \binom{n + l}{2l + 1} \frac{(-\alpha)^{l+1}}{p_2^l} .
\end{align*}
Combining the latter identity with \eqref{eq:det_Aa_i} for $i=n-1$ in \eqref{eq:solution_discr_exp1} and observing that $p_1 = \sqrt{D_ap_2}$ lead to \eqref{eq:solution_discr}.\qed
\end{pf}
Observe that Theorem \ref{theorem:design_k1_opti} provides two expressions for the solution $\kappa_*^{(n)}$ of the $\mathbf{S_{\alpha}^{(n)}}$ problem, namely \eqref{eq:solution_discr} and \eqref{eq:solution_discr_exp1}. The former has been derived from the latter in order to get a computable expression that is easier to use in the convergence analysis that is performed hereafter.   
\begin{lem}\label{lemma:conv_sol}
    Let~$\alpha\in\mathbb{R}$ be such that $0< \alpha<\pi^2 D_a/(4L^2)$. The solution $\kappa_*^{(n)}$ of the $\mathbf{S_{\alpha}^{(n)}}$ problem~\eqref{eq:pbm_opti_discr} converges to $\kappa_*$ given by~
    \begin{equation}\label{eq:solution}
        \kappa_* :=  -\sqrt{\alpha D_a^{-1}}\tan{(L\sqrt{\alpha D_a^{-1}})},
    \end{equation}
    i.e.  
    \begin{equation}\label{eq:convergence}
     \kappa_*^{(n)}\mathbb{1}_{\left(0,-\alpha_f^{(n)}\right)}(\alpha) \rightarrow   \kappa_*\quad \text{as}\quad n\rightarrow \infty\,.
    \end{equation}
\end{lem}
\begin{pf}
        Consider the sequences $(u_n)$ and $(v_n)$ which are defined, for all $n\in\mathbb{N}\setminus\{0,1\}$, by
        \begin{equation*}
            \left\{
            \begin{aligned}
                u_n &:= \sum_{l=0}^{n-1} \frac{1}{(2l+1)!}\frac{(-\alpha)^{l+1}L^{2l+1}}{D_a^{l+\frac{1}{2}}},\\
                v_n &:= \sum_{l=0}^{n-1} \binom{n+l}{2l+1}\frac{(-\alpha)^{l+1}}{p_2^{l+\frac{1}{2}}},
            \end{aligned}
            \right.
        \end{equation*}
        where $v_n$ appears in the numerator of $\kappa_*^{(n)}$, defined by \eqref{eq:solution_discr}. Appealing to the Maclaurin series of $\sin(\cdot)$ leads to $u_n\rightarrow -\sqrt{\alpha}\sin(L\sqrt{\alpha D_a^{-1}})$ as $n\rightarrow\infty$. In Appendix C, it is shown that $|u_n-v_n|\rightarrow 0$ as $n\rightarrow \infty$. Thus,
        \begin{equation}\label{eq:conv_vn}
            v_n\rightarrow -\sqrt{\alpha}\sin(L\sqrt{\alpha D_a^{-1}}) \quad \text{as}\quad n\rightarrow\infty.
        \end{equation}
        By arguments similar to those used to prove \eqref{eq:conv_vn}, it can be shown that 
        \begin{equation}\label{eq:conv_deno}
            {\displaystyle\sum_{l=0}^{n-1} \binom{n+l-1}{2l}\frac{(-\alpha)^{l}} {p_2^{l}}}\rightarrow \cos(L\sqrt{\alpha D_a^{-1}})\quad\text{as}\quad n\rightarrow\infty.
        \end{equation}
        Observe now that, since $0<\alpha<\pi^2 D_a/(4L^2)$ and $-\alpha_f^{(n)}\rightarrow \pi^2 D_a/(4L^2)$ as $n\rightarrow\infty$, there exists $N_{\alpha}\in\mathbb{N}$ such that 
        \begin{equation}\label{eq:cdt_alpha}
            -\alpha_f^{(n)}>\alpha \quad \forall\:n\geq N_{\alpha},
        \end{equation}
        and so, by Theorem \ref{theorem:design_k1_opti}, for all $n\geq N_{\alpha}$, the optimization problem \eqref{eq:pbm_opti_discr} is feasible and the denominator of \eqref{eq:solution_discr} is nonzero. Combining the latter property with \eqref{eq:conv_vn}, \eqref{eq:conv_deno}, \eqref{eq:cdt_alpha}, and the fact that $\cos(L\sqrt{\alpha D_a^{-1}})\neq 0$, we conclude that \eqref{eq:convergence} holds.
\qed 
\end{pf}
As was hopefully expected, it turns out that the asymptotic value obtained in Lemma \ref{lemma:conv_sol} is actually the solution of the $\mathbf{S_{\alpha}}$ problem.
\begin{thm}\label{theorem:design_kappa_opti}
    Fix~$\alpha>0$. The $\mathbf{S_{\alpha}}$ problem~\eqref{eq:pbm_opti} for the diffusion system \eqref{eq:pure_diffusion}-\eqref{eq:cdb_control} with a state feedback of the form~\eqref{eq:state_feedback_pure_diff} is feasible if, and only if, 
    $\alpha <\pi^2 D_a/(4L^2)$. In this case, $\mathbf{S_{\alpha}}$ has a unique solution $\kappa_*$ given by \eqref{eq:solution}.
\end{thm}
\begin{pf} 
    Recall that the semigroup $(T_{A_{\kappa}}(t))_{t\geq0}$ is positive for every $\kappa\in\mathbb{R}$. The feasibility property of $\mathbf{S_{\alpha}}$ is therefore based on condition~\eqref{eq:pbm_opti_const_stability}. The parameter $\alpha$ being positive, a necessary condition for the latter condition is $\kappa<0$. So, fix $\kappa<0$. A straightforward spectral analysis shows that the eigenvalues of $A_{\kappa}$ are precisely the real zeros of the function $f_{\kappa}$, where, for any $\eta \in {\mathbb R}$, $f_{\eta} : {\mathbb R_-}\setminus\{-\frac{D_a}{4L^2}(\pi+2k\pi)^2:k\in\mathbb{N} \}\rightarrow \mathbb{R}$ is defined by
    \begin{equation*}
        f_{\eta}(\lambda) := \eta + \sqrt{-\lambda D_a^{-1}}\tan(L\sqrt{-\lambda D_a^{-1}}).
    \end{equation*} 
    Moreover, $f_{\kappa}$ is continuous and strictly decreasing on $(-\pi^2D_a/(4L^2),0]$, and satisfies $\lim_{\lambda \searrow \frac{-\pi^2D_a}{4L^2}} f_{\kappa}(\lambda) = +\infty$ and $f_{\kappa}(0) = \kappa.$
    Using arguments similar to those in the proof of Lemma~\ref{lemma:feasability_practicality}, with $f_{\kappa}$ and $-\pi^2D_a/(4L^2)$ replacing $g_{\kappa}$ and $\alpha_f^{(n)}$, respectively, we obtain the following results: $\mathbf{S_{\alpha}}$ is feasible if, and only if, $\alpha <\pi^2 D_a/(4L^2)$, and, in that case, for all $\kappa<0$, \eqref{eq:pbm_opti_const_positivity} and \eqref{eq:pbm_opti_const_stability} hold if, and only if,
    \begin{equation}\label{eq:design_kappa}
        \kappa \leq -\sqrt{\alpha D_a^{-1}}\tan(L\sqrt{\alpha D_a^{-1}}).
    \end{equation}
    If $0<\alpha <\pi^2 D_a/(4L^2)$, then the right-hand side of \eqref{eq:design_kappa} is negative and the latter equivalence is actually valid for all $\kappa\in\mathbb{R}$. Thus, the $\mathbf{S_{\alpha}}$ problem has a unique solution $\kappa_*$ given by \eqref{eq:solution}.\qed
\end{pf}
From Theorem~\ref{theorem:design_kappa_opti}, it follows that it is not possible to design a positively exponentially stabilizing state feedback of the form~\eqref{eq:state_feedback_pure_diff} for the diffusion system \eqref{eq:pure_diffusion}-\eqref{eq:cdb_control} with a stability margin greater or equal to $\pi^2 D_a/(4L^2)$. On the other hand, if~$0<\alpha <\pi^2 D_a/(4L^2)$, then $|\kappa_*|$ is the smallest value of $|\kappa|$ that yields a positively exponentially stabilizing state feedback of the form~\eqref{eq:state_feedback_pure_diff} with a stability margin greater or equal to $\alpha$. 

\section{Positive State Observers}\label{section:estimation}
In Section~\ref{section:exact_feedback}, we designed exact state feedbacks of the form $u(t)=\kappa x(0,t)$ to achieve positive exponential stabilization of the diffusion system \eqref{eq:pure_diffusion}-\eqref{eq:cdb_control}. However, implementing such feedback laws requires the state $x(\cdot,t)$ to be sufficiently well known at each time $t\geq0$. Here, only the output $y(t)=x(L,t)$ is available (via sensors, in practice). Consequently, it becomes necessary to be able to estimate the full state by means of the available output. This motivates the design of a state observer for the uncontrolled (that is, $u=0$) diffusion system \eqref{eq:pure_diffusion}-\eqref{eq:cdb_control}. In this section, in addition to guaranteeing the asymptotic convergence of the estimation error, the goal is to design state observers that preserve the nonnegativity of the estimated state, in order to comply with the positivity of the nominal system.

To this end, let us consider the following state observer for the uncontrolled system \eqref{eq:pure_diffusion}-\eqref{eq:output}:
\begin{equation} \label{eq:pure_diffusion_observer}
    \frac{\partial \tilde{x}}{\partial t}(z,t) = D_a\frac{\partial^2 \tilde{x}}{\partial z^2}(z,t), \quad z \in [0,L], \: t \geq 0,
\end{equation}
with the boundary conditions
\begin{equation} \label{eq:observer_BC}
        \frac{\partial \tilde{x}}{\partial z}(0,t) = 0, \quad
        \frac{\partial \tilde{x}}{\partial z}(L,t) - \gamma(\tilde{y}(t) - y(t)) = 0,
\end{equation}
where the measured output $y$ is given by \eqref{eq:output}, and where
\begin{equation}\label{eq:estimated_output}
    \tilde{y}(t):=\tilde{x}(L,t),\quad t\geq 0,
\end{equation}
denotes the estimated output, and $\gamma \in \mathbb{R}$ is the observer gain, to be designed. The dynamics of the state estimation error $e:=\tilde{x}-x$ is described by the following PDE:
\begin{equation} \label{eq:error_dynamics}
    \frac{\partial e}{\partial t}(z,t) = D_a\frac{\partial^2 e}{\partial z^2}(z,t), \quad z \in [0,L], \: t \geq 0,
\end{equation}
with the boundary conditions
\begin{equation} \label{eq:error_BC}
        \frac{\partial e}{\partial z}(0,t) = 0, \quad
        \frac{\partial e}{\partial z}(L,t) - \gamma e(L,t) = 0.
\end{equation}
By considering the invertible map $S : \mathrm{H}^2(0,L)\to \mathrm{H}^2(0,L)$ given by $(Sx)(z) = x(L-z)$, we notice that the estimation error dynamics \eqref{eq:error_dynamics}-\eqref{eq:error_BC} is equivalent to the closed-loop system analyzed in Section~\ref{section:exact_feedback}, where $\gamma$ replaces the control gain $\kappa$. Consequently, the estimation error dynamics~\eqref{eq:error_dynamics}-\eqref{eq:error_BC} is exponentially stable, or equivalently the operator 
$SA_{\gamma}S^{-1}$, with domain $D(SA_{\gamma}S^{-1})$, is the infinitesimal generator of an exponentially stable $C_0$-semigroup on $\mathrm{L}^2(0,L)$ if, and only if, $\gamma < 0$; see \cite[Theorem 4]{Dehaye2016}. 
Moreover the estimation error dynamics~\eqref{eq:error_dynamics}-\eqref{eq:error_BC} is positive for all $\gamma\in\mathbb{R}$.

Beyond stability of the estimation error dynamics, we aim at ensuring the nonnegativity of any estimated state trajectory. For notational simplicity, let $\mathcal{A}$ denote the following operator on $\mathrm{L}^2(0,L) \times \mathrm{L}^2(0,L)$
    \begin{equation*}
        \mathcal{A} = \begin{pmatrix}
            D_a({\rm  d}^2/{\rm  d}  z^2) & 0\\
            0 & D_a({\rm  d}^2/{\rm  d}  z^2)
        \end{pmatrix}.
    \end{equation*}
\begin{prop}\label{prop:positivity_A0gamma}
    For all $\gamma< 0$, the operator $A_{0\gamma}$ defined by $A_{0\gamma}X=\mathcal{A}X$ for $X\in D(A_{0\gamma})$, where  \begin{equation*}\label{eq:operator_A0gamma}
    \begin{aligned}
            &D(A_{0\gamma}) = \left\{X=(x,\tilde{x})\in \mathrm{H}^2(0,L) \times \mathrm{H}^2(0,L):\right.\\ &\dfrac{{\rm  d}}{{\rm  d}  z}\begin{pmatrix}x\\\tilde{x}\end{pmatrix}(0) = 0,\left.\dfrac{{\rm  d}}{{\rm  d}  z}\begin{pmatrix}x\\\tilde{x}\end{pmatrix}(L) = \begin{pmatrix}
            0&0\\-\gamma&\gamma\end{pmatrix}\begin{pmatrix}x\\\tilde{x}\end{pmatrix}(L) \right\},  
    \end{aligned}
    \end{equation*}
    generates a positive $C_0$-semigroup $(T_{A_{0\gamma}}(t))_{t\geq 0}$ on $\mathrm{L}^2(0,L) \times \mathrm{L}^2(0,L)$. 
\end{prop}
\begin{pf}
    Let $\gamma< 0$ and consider the resolvent operator $R(\lambda,A_{0\gamma})$.  
    By \cite[Corollary 11.4]{Batkai2017}, we have to show that $R(\lambda,A_{0\gamma})\geq0$ for large real $\lambda$. So, fix $\lambda$ sufficiently large. 
    By using standard arguments, one gets that, for all $X:=(x \hspace{0.2cm}\tilde{x})^{\top}\in \mathrm{L}^2(0,L)\times \mathrm{L}^2(0,L)$ and $z\in[0,L]$,
    \begin{equation*}
            (R(\lambda,A_{0\gamma})X)(z) = \begin{bmatrix}
           \int_0^L \frac{G_1(z,\zeta)}{D_a}x(\zeta)\,\rm{d}\zeta \\
           \int_0^L \frac{G_2(z,\zeta)}{D_a}x(\zeta)+\frac{G_3(z,\zeta)}{D_a}\tilde{x}(\zeta)\,\rm{d}\zeta 
        \end{bmatrix},
    \end{equation*}
    where, with $\displaystyle \mu := \sqrt{\lambda D_a^{-1}}$, for all $\zeta\in[0,L]$,
    \begin{equation*}
        \left\{
        \begin{aligned}
            G_1(z,\zeta) &:= \frac{f(\zeta)\cosh(\mu z)}{\mu\sinh(\mu L)} - \frac{g(z,\zeta)}{\mu} \mathbb{1}_{[0,z]}(\zeta), \\
            G_2(z,\zeta) &:= \frac{-\gamma\cosh(\mu \zeta)}{\mu\sinh(\mu L)}h(z),\\
            G_3(z,\zeta) &:= (f(\zeta)-\frac{\gamma}{\mu} g(L,\zeta))h(z)- \frac{g(z,\zeta)}{\mu} \mathbb{1}_{[0,z]}(\zeta),
        \end{aligned}
        \right.
    \end{equation*}
    and $f(\zeta) := \cosh(\mu(L-\zeta))$, $g(z,\zeta) := \sinh(\mu(z-\zeta))$, $h(z) := \cosh(\mu z)(\mu\sinh(\mu L)-\gamma\cosh(\mu L))^{-1}$.
    Since $\gamma<0$, by using standard trigonometric identities, we find that $G_i(z,\zeta)\geq0$ for all $i\in\{1,2,3\}$ and $z,\zeta\in[0,L]$. Consequently, if $X=(x \hspace{0.2cm}\tilde{x})^{\top}\in (\mathrm{L}^2(0,L)\times \mathrm{L}^2(0,L))_+$, then $R(\lambda,A_{0\gamma})X$ is nonnegative.\qed
\end{pf}
A consequence of Proposition \ref{prop:positivity_A0gamma} is that, by combining the state dynamics \eqref{eq:pure_diffusion}-\eqref{eq:cdb_control} (with $u=0$) and the state observer dynamics \eqref{eq:pure_diffusion_observer}-\eqref{eq:observer_BC} with $\gamma<0$, the resulting system is positive. 
Thus, for every nonnegative initial state $x_0$ and every nonnegative estimated initial state $\tilde{x}_0$ such that $(x_0 \hspace{0.2cm}\tilde{x}_0)^{\top}\in D(\mathcal{A}_{\gamma})$, the associated estimated state trajectory $\tilde{x}$ is nonnegative when $\gamma<0$. Therefore, the proposed state observer \eqref{eq:pure_diffusion_observer}-\eqref{eq:observer_BC} is well-suited to positive state estimation if, and only if, $\gamma < 0$.

To tune the observer gain $\gamma$ in \eqref{eq:pure_diffusion_observer}-\eqref{eq:observer_BC}, we now introduce a dual problem to $\mathbf{S_{\alpha}}$. More precisely, let us consider, for any $\alpha>0$, the optimal state estimation problem for the diffusion system, denoted $\mathbf{E_{\alpha}}$, given by
\begin{subequations}\label{eq:pbm_opti_est}
    \begin{align}
        \displaystyle \min_{\gamma \,\in \,\mathbb{R}}\quad &\gamma^2, \label{eq:pbm_opti_est_obj}\\
        \text{s.t.}\quad & (T_{A_{0\gamma}}(t))_{t\geq0} \text{ is positive,} \label{eq:pbm_opti_est_const_positivity}\\
        \quad & (T_{SA_{\gamma}S^{-1}}(t))_{t\geq0} \text{ is } \beta\text{-exponentially stable}, \notag\\
        \quad &\,\,\forall \beta\in(-\alpha,0]\,.\label{eq:pbm_opti_est_const_stability}
    \end{align}
\end{subequations}
\begin{rem}\label{rem:motivation_pbm_est_opti}
A state observer of the form~\eqref{eq:pure_diffusion_observer}-\eqref{eq:observer_BC} where $\gamma$ satisfies~\eqref{eq:pbm_opti_est_const_positivity} and~\eqref{eq:pbm_opti_est_const_stability} may be viewed as a constrained high gain (static) state observer. The high gain interpretation comes from the fact that the output error injection term $\gamma(\tilde{y}(t) - y(t))$ is parametrized by a single scalar $\gamma$ and can become arbitrarily large as $|\gamma|\rightarrow \infty$. This kind of observer is widely used in the literature: see \cite{Khalil2008} and references therein or \cite{Kitsos2021}. In these references, it is mentioned that a key feature of high gain type observers is a fast convergence of the estimation error. Here, this effect is sought with condition \eqref{eq:pbm_opti_est_const_stability} and the choice of a sufficiently large $\alpha$ (in the following, we will see that $\alpha$ cannot be driven too high). However, high gain type observers may be subject to the peaking phenomenon; see \cite{Khalil2008}. Numerical simulations that we performed  for the state observer \eqref{eq:pure_diffusion_observer}-\eqref{eq:observer_BC} reveal also this phenomenon. In this context, the optimality criterion~\eqref{eq:pbm_opti_est_obj} aims to prevent this effect.
\end{rem}

Remark \ref{rem:motivation_pbm_est_opti} highlights the fact that $\mathbf{E_{\alpha}}$ represents a trade-off between the advantage offered by a sufficiently high gain state observer (in terms of decay rate of the estimation error dynamics) and the attenuation of the peaking phenomenon, while ensuring nonnegativity of the estimated state trajectories. The solution of $\mathbf{E_{\alpha}}$ is given in the next theorem and follows directly from Theorem~\ref{theorem:design_kappa_opti}.
\begin{thm}\label{theorem:design_gamma_opti}
    Fix~$\alpha>0$. The $\mathbf{E_{\alpha}}$ problem~\eqref{eq:pbm_opti} for the diffusion system \eqref{eq:pure_diffusion}-\eqref{eq:cdb_control} (with $u=0$) is feasible if, and only if, $\alpha <\pi^2 D_a/(4L^2)$. In this case, $\mathbf{E_{\alpha}}$ has a unique solution $\gamma_*$, which satisfies $\gamma_*=\kappa_*$ with $\kappa_*$ given by \eqref{eq:solution}.
\end{thm}
\begin{pf}
    Since the $C_0$-semigroup $(T_{A_{0\gamma}}(t))_{t\geq0}$ is positive for all $\gamma<0$ (see Proposition \ref{prop:positivity_A0gamma}), and condition \eqref{eq:pbm_opti_est_const_stability} implies that $\gamma<0$, the feasible set of $\mathbf{E_{\alpha}}$ is equal to the set of the values of $\gamma$ such that \eqref{eq:pbm_opti_est_const_stability} is satisfied. Thus the feasible set of $\mathbf{E_{\alpha}}$ coincides with the one of $\mathbf{S_{\alpha}}$ in Section~\ref{section:exact_feedback}. Consequently, $\mathbf{E_{\alpha}}$ and $\mathbf{S_{\alpha}}$ are equivalent. The statement then follows from~Theorem \ref{theorem:design_kappa_opti}.\qed
\end{pf}
The same type of interpretations in terms of dynamics as those related to Theorem~\ref{theorem:design_kappa_opti} can be adapted to Theorem~\ref{theorem:design_gamma_opti} and to the state estimation context. Moreover, observe that the feasibility conditions of $\mathbf{E_{\alpha}}$, i.e. \eqref{eq:pbm_opti_est_const_positivity}-\eqref{eq:pbm_opti_est_const_stability}, can be interpreted as sufficient conditions for the $\beta$-exponential detectability for all $\beta\in(-\alpha,0]$ of the considered uncontrolled diffusion system with the boundary observation \eqref{eq:output}; see \cite[Section 8.1]{Curtain_Zwart2020}.\\
By using equivalence arguments in the same vein as in Theorem~\ref{theorem:design_gamma_opti}, we obtain  corresponding results for the discretized state dynamics~\eqref{eq:pure_diffusion_discr_control} (with $u=0$) combined with a spatially discretized (using an analogous finite difference scheme) version of the state observer~\eqref{eq:pure_diffusion_observer}-\eqref{eq:observer_BC}. The solution of the $\mathbf{E_{\alpha}}$ problem can also be recovered via a discretization approach very similar to the one used in Section \ref{section:exact_feedback} but directly adapted to the optimal state estimation problem. See~\cite{PiengeonWinkin2026} for more detail.

\section{Observer-based positive stabilization}\label{section:observer_based_positive_stabilization}
In Section~\ref{section:exact_feedback}, we have shown how exact state feedbacks can be designed to ensure exponential stabilization of the diffusion system~\eqref{eq:pure_diffusion}-\eqref{eq:cdb_control} while preserving the nonnegativity of any state trajectory. However, in practice, the full state $x(\cdot,t)$ is not available and typically the state value is known only at a finite number of points or only the mean value of the state on a small subinterval is known. Here, we assume that the measurement is made at the boundary~\eqref{eq:output}. Therefore, in place of the exact state feedback \eqref{eq:state_feedback_pure_diff}, consider the observer-based feedback law for the diffusion system~\eqref{eq:pure_diffusion}-\eqref{eq:cdb_control} given by
\begin{equation}\label{eq:observer_feedback_pure_diff}
    u(t) = \kappa \tilde{x}(0,t)\,, \quad t\geq 0\,,
\end{equation}
where $\kappa \in \mathbb{R}$ denotes the feedback gain and $\tilde{x}(\cdot,t)$ is the estimated state generated by the observer dynamics~\eqref{eq:pure_diffusion_observer} with the boundary conditions
\begin{equation} \label{eq:observer_BC_control}
        \frac{\partial \tilde{x}}{\partial z}(0,t) + u(t)= 0, \quad
        \frac{\partial \tilde{x}}{\partial z}(L,t) - \gamma(\tilde{y}(t) - y(t)) = 0,
\end{equation}
where $y$ (resp. $\tilde{y}$) is the measured (resp. estimated) output~\eqref{eq:output} (resp.~\eqref{eq:estimated_output}) and $\gamma \in \mathbb{R}$ is the observer gain. The dynamics of the state estimation error $e = \tilde{x}-x$ is still governed by \eqref{eq:error_dynamics}-\eqref{eq:error_BC}. Through this whole section, we will make the following \textbf{Standing Assumption:} 
$\kappa \; \neq \; \gamma$.

\subsection{Nonpositivity of the $(x,\tilde{x})$-dynamics}\label{subsection:Nonpositivity}
The objective is to determine whether exponential stabilization of the diffusion system can be achieved via the observer-based state feedback \eqref{eq:observer_feedback_pure_diff} while ensuring the nonnegativity of the state and the estimated state. First, we focus on exponential stability. The combination of the state dynamics and the estimation error dynamics in terms of $x$ and $e$ can be given an infinite-dimensional state-space description of the form
\begin{equation}\label{eq:system_x_e}
    \begin{pmatrix}
        \dot{x}(t)\\\dot{e}(t)
    \end{pmatrix}
    = 
    \tilde{A}_{\kappa\gamma}\begin{pmatrix}
        x(t)\\e(t)
    \end{pmatrix}, 
\end{equation}
where $\tilde{A}_{\kappa\gamma}X=\mathcal{A}X$ for $X\in D(\tilde{A}_{\kappa\gamma})$ and
\begin{equation*}
    \begin{array}{rcl}
        D(\tilde{A}_{\kappa\gamma}) &=& \left\{X:=(x\hspace{0.2cm}e)^{\top}\in \mathrm{H}^2(0,L) \times \mathrm{H}^2(0,L):\right.\\ &&\dfrac{{\rm  d}}{{\rm  d}  z}\begin{pmatrix}x\\e\end{pmatrix}(0) = \begin{pmatrix}
        -\kappa&-\kappa\\0&0\end{pmatrix}\begin{pmatrix}x\\e\end{pmatrix}(0),\\ &&\left.\dfrac{{\rm  d}}{{\rm  d}  z}\begin{pmatrix}x\\e\end{pmatrix}(L) = \begin{pmatrix}
        0&0\\0&\gamma\end{pmatrix}\begin{pmatrix}x\\e\end{pmatrix}(L) \right\}.
    \end{array}
\end{equation*}
\begin{prop}\label{prop:stability_Atilde_kappagamma}
    The operator $\tilde{A}_{\kappa\gamma}$ generates an exponentially stable $C_0$-semigroup $(T_{\tilde{A}_{\kappa\gamma}}(t))_{t\geq 0}$ on $\mathrm{L}^2(0,L) \times \mathrm{L}^2(0,L)$ if, and only if, $\kappa<0$ and $\gamma<0$.
\end{prop}
\begin{pf}
    By using 
    \cite[Lemma 3.2.9]{Curtain_Zwart2020}, it can be shown that, for all $\kappa,\gamma<0$, $\tilde{A}_{\kappa\gamma}$ is a Riesz spectral operator. Moreover, by a straightforward computation of the eigenvalues of $\tilde{A}_{\kappa\gamma}$, one gets that $\sigma(\tilde{A}_{\kappa\gamma}) = \sigma(A_{\kappa}) \cup \sigma(A_{\gamma})$. As mentioned in \cite[Section 5.1]{Hastir2023}, the eigenvalues of $A_{\kappa}$ and $A_{\gamma}$ are real, hence $\sigma(\tilde{A}_{\kappa\gamma})\subseteq\mathbb{R}$. Moreover it is upper bounded. It follows, by \cite[Theorem 3.2.8]{Curtain_Zwart2020}, that $\tilde{A}_{\kappa\gamma}$ generates a $C_0$-semigroup,  
    which is exponentially stable if, and only if,
    $\sigma_p(A_{\kappa})$ and $\sigma_p(A_{\gamma})$ are included in $\mathbb{R}_-\setminus\{0\}$, or equivalently $\kappa<0$ and $\gamma<0$.\qed
\end{pf}
By Proposition \ref{prop:stability_Atilde_kappagamma}, the $(x,\tilde{x})$-dynamics is also exponentially stable if, and only if, $\kappa,\gamma<0$, since the latter may be obtained by a change of variables with respect to the $(x,e)$-dynamics. However, if $\kappa,\gamma<0$, then the $(x,\tilde{x})$-dynamics is not positive, as shown below.
\begin{prop}\label{prop:non_positivity_A_kappagamma}
    For all $\kappa,\gamma<0$, the $C_0$-semigroup $(T_{A_{\kappa\gamma}}(t))_{t\geq 0}$ generated by $A_{\kappa\gamma}$ defined as $A_{\kappa\gamma}X=\mathcal{A}X$ for $X\in D(A_{\kappa\gamma})$, where 
    \begin{equation*}
    \begin{array}{rcl}
            D(A_{\kappa\gamma}) &=& \left\{X:=(x\hspace{0.2cm}\tilde{x})^{\top}\in \mathrm{H}^2(0,L) \times \mathrm{H}^2(0,L):\right.\\ &&\dfrac{{\rm  d}}{{\rm  d}  z}\begin{pmatrix}x\\\tilde{x}\end{pmatrix}(0) = \begin{pmatrix}
            0&-\kappa\\0&-\kappa\end{pmatrix}\begin{pmatrix}x\\\tilde{x}\end{pmatrix}(0),\\ &&\left.\dfrac{{\rm  d}}{{\rm  d}  z}\begin{pmatrix}x\\\tilde{x}\end{pmatrix}(L) = \begin{pmatrix}
            0&0\\-\gamma&\gamma\end{pmatrix}\begin{pmatrix}x\\\tilde{x}\end{pmatrix}(L) \right\}, 
    \end{array}
    \end{equation*}
    is not positive.
\end{prop}
\begin{pf}
    Let $\kappa,\gamma< 0$. First, note that $A_{\kappa\gamma}$ generates an (exponentially stable) $C_0$-semigroup $(T_{A_{\kappa\gamma}}(t))_{t\geq 0}$ on $\mathrm{L}^2(0,L) \times \mathrm{L}^2(0,L)$ because so does $\tilde{A}_{\kappa\gamma}$.
    Using a routine computation, one gets that  
    the resolvent operator of $A_{\kappa\gamma}$ at zero 
    is given, for all $X:=(x\hspace{0.2cm}\tilde{x})^{\top}\in \mathrm{L}^2(0,L)\times \mathrm{L}^2(0,L)$ and $z\in[0,L]$, by
    \begin{align*}
            &(R(0,A_{\kappa\gamma})X)(z)\\
            &= \begin{pmatrix}
           \int_0^L  f(z,\zeta)x(\zeta)+ (\frac{1}{\gamma}-L+\zeta)\tilde{x}(\zeta) \,\rm{d}\zeta \\
           \int_0^L (z-\frac{1}{\kappa})x(\zeta) -(z-\zeta)\mathbb{1}_{[0,z]}(\zeta)\tilde{x}(\zeta)  \,\rm{d}\zeta 
        \end{pmatrix},
    \end{align*}
    where $\displaystyle  f(z,\zeta) := z-\frac{\kappa+\gamma}{\kappa\gamma}+L-\zeta-(z-\zeta)\mathbb{1}_{[0,z]}(\zeta),$
    for all $\zeta\in[0,L]$. Now  consider the constant function $ \bar{X}:[0,L]\rightarrow\mathbb{R}^2$ given by 
    \begin{equation}\label{eq:counter_example}
            \bar{X}(z) :=-
    \begin{pmatrix}
        \kappa &\:\: \kappa+\gamma
    \end{pmatrix}^{\top} 
    \end{equation}
    and observe that $\bar{X}\in (\mathrm{L}^2(0,L) \times \mathrm{L}^2(0,L))_+$. Evaluating $R(0,A_{\kappa\gamma})\bar{X}$ at zero
    yields $(R(0,A_{\kappa\gamma})\bar{X})(0) = (\frac{\gamma L^2}{2}\hspace{0.2cm}L)^{\top}$
    and so, by a continuity argument,  the first component of $R(0,A_{\kappa\gamma})\bar{X}$ is negative on a small interval $[0, \eta ] \subset [0,L].$ Hence, $R(0,A_{\kappa\gamma})\bar{X}
    \notin (\mathrm{L}^2(0,L) \times \mathrm{L}^2(0,L))_+$.
    The conclusion follows by \cite[Corollary 11.4]{Batkai2017}.\qed
\end{pf}
A consequence of Propositions \ref{prop:stability_Atilde_kappagamma} and \ref{prop:non_positivity_A_kappagamma} is that, while exponential stability of the $(x,e)$-dynamics \eqref{eq:system_x_e} is ensured with $\kappa,\gamma<0$, positivity of both $x$ and $\tilde{x}$ cannot be guaranteed simultaneously for all nonnegative initial conditions $(x_0\hspace{0.2cm}\tilde{x}_0)^{\top}\in D(A_{\kappa\gamma})$. This kind of property is known for infinite-dimensional linear systems with bounded control and observation operators,  
see \cite[Theorem 3.1]{Binid2021}. Here such result is extended to a diffusion model with boundary control and observation.

\subsection{Positively stabilizing observer-based state feedbacks}\label{subsection:Cones_design}
As a trade-off, the idea developed in what follows consists in removing the nonnegativity constraint of the estimated state trajectory $\tilde{x}$ and looking for conditions under which only the state trajectory $x(\cdot,t)$ remains nonnegative. In the end, this is the main design objective in a control context, together with the exponential stability of the $(x,e)$-dynamics.

To find the desired conditions, we focus on the $(x,e)$-dynamics because one of the two parts, i.e., the dynamics of the error, is independent of the other part, i.e, the dynamics of the state. This property will be helpful in the following. Observe now that the state trajectory $x$, issued from the $(x,e)$-dynamics, is entirely determined by $\kappa$, $\gamma$, the initial state $x_0$, and the initial state estimation error $e_0$. In a state estimation context, $x_0$ is supposed to be unknown, and so no condition (apart from nonnegativity) should be imposed on $x_0$.
Moreover, for all $\kappa,\gamma<0$, we can always find an initial estimation error $e_0$ and a nonnegative initial state $x_0$ from which the state trajectory becomes negative 
at some time. Indeed, using some arguments of the proof of Lemma \ref{lemma:conditions_positive_state_PDE} below, we can show that it suffices to take $x_0=0$ and $e_0\in L^2(0,L)_+\setminus\{0\}$.
Therefore, in the following, the goal is to find, for all $\kappa,\gamma<0$, which initial estimation error $e_0$ (and therefore which $\tilde{x}_0$) can be committed such that, whatever the nonnegative initial state $x_0$ that is being observed, the resulting state trajectory $x$ is nonnegative. Such $e_0$ will be called "suitable" in the following.  

It is straightforward to show that any nonpositive $e_0$ is suitable for all $\kappa,\gamma<0$; see Lemma \ref{lemma:conditions_positive_state_PDE}. However, this nonpositivity condition is very restrictive. Indeed, recall that the estimation error dynamics is positive for all $\gamma\in\mathbb{R}$ (see Section \ref{section:estimation}), and so taking $e_0\leq0$ amounts to underestimating the full state, that is $\tilde{x}(t)\leq x(t)$ for all $t\geq0$. Starting from this observation, in the following, we will show how to construct nontrivial initial estimation errors for which the nonnegativity property of the state trajectory $x$, issued from the $(x,e)$-dynamics, is satisfied. 

Once again, we will first consider this problem for a discretized version of the $(x,e)$-dynamics. Applying the same finite difference schemes than in Sections \ref{section:model} and \ref{section:exact_feedback} to \eqref{eq:pure_diffusion}-\eqref{eq:cdb_control} combined with \eqref{eq:error_dynamics}-\eqref{eq:error_BC} and $u$ given by \eqref{eq:observer_feedback_pure_diff}, leads to 
\begin{equation}\label{eq:x_e_dyn_discr}
    \left\{
    \begin{aligned}
        \dot{x}^{(n)}(t) &= A^{(n)}_{\kappa}x^{(n)}(t)+b^{(n)}\kappa {\mathbb{e}_1^{(n)}}^{\top} e^{(n)}(t), \\
        \dot{e}^{(n)}(t) &= S^{(n)}A^{(n)}_{\gamma}S^{(n)}e^{(n)}(t),
    \end{aligned}
    \right.
\end{equation}
where $S^{(n)}={S^{(n)}}^{-1}$ is the permutation matrix with $1$’s on the anti-diagonal and $0$’s elsewhere, and the state trajectories of~\eqref{eq:x_e_dyn_discr}
are defined by $x^{(n)}(\cdot) := (x(z_1,\cdot) \, \cdots \, x(z_n,\cdot))^{\top} \in {\mathbb R}^n,$ and $e^{(n)}(\cdot) := (e(z_1,\cdot) \, \cdots \, e(z_n,\cdot))^{\top} \in {\mathbb R}^n$.
The associated estimated state trajectory $\tilde{x}^{(n)}$ is defined by $\tilde{x}^{(n)}:= x^{(n)}+e^{(n)}$, and so \eqref{eq:x_e_dyn_discr} is equivalent to 
\begin{equation}\label{eq:x_xtilde_dyn_discr}
    \left\{
    \begin{aligned}
        &\dot{x}^{(n)}(t) = A^{(n)}x^{(n)}(t)+b^{(n)}\kappa {\mathbb{e}_1^{(n)}}^{\top} \tilde{x}^{(n)}(t), \\
         &\begin{aligned}\dot{\tilde{x}}^{(n)}(t)=(A^{(n)}_{\kappa}+&\gamma p_1\mathbb{e}_n^{(n)}{c^{(n)}}^{\top})\tilde{x}^{(n)}(t) \\ &-\gamma p_1{\mathbb{e}_n^{(n)}}{c^{(n)}}^{\top}x^{(n)}(t).
        \end{aligned}
    \end{aligned}
    \right.
\end{equation}

Properties that are similar to those established for the nominal $(x,e)$ and $(x,\tilde{x})$-dynamics hold. Indeed, in view of  Remark \ref{rem:A_kappa_property}, the $(x^{(n)},e^{(n)})$-dynamics \eqref{eq:x_e_dyn_discr} is exponentially stable if, and only if, $\kappa,\gamma<0$, and the $(x^{(n)},\tilde{x}^{(n)})$-dynamics \eqref{eq:x_xtilde_dyn_discr} is positive if, and only if, $\kappa \geq 0$ and $\gamma \leq 0$. Hence, as for the nominal PDE model, it is impossible to positively exponentially stabilize the discretized diffusion system thanks to an observer-based state feedback of the form $u(t) = \kappa {\mathbb{e}_1^{(n)}}^{\top} \tilde{x}^{(n)}(t)$ related to a positive state observer. Note that this result is actually valid for any finite-dimensional LTI positive system with an observer-based feedback; see~\cite[Theorem 5.1]{Rami2011}. 

As a compromise, we now remove the nonnegativity constraint of the estimated state $\tilde{x}^{(n)}$. We can easily show that, for all $\kappa, \gamma<0$, the set of initial state estimation errors $e_0^{(n)}$ leading to a nonnegative state trajectory $x^{(n)}$ of \eqref{eq:x_e_dyn_discr}, includes $\mathbb{R}_-^n$. However, there are many other suitable $e_0^{(n)}$, as shown in the next result. Note that Lemma \ref{lemma:conditions_positive_state} is instrumental for Theorem \ref{theorem:error_inP_for_positivity}$\beta$.
\begin{lem}\label{lemma:conditions_positive_state}
Let $\gamma\in\mathbb{R}$ and consider the cone 
\begin{align}
    \mathcal{E}^{(n)}:=\bigg\{&e^{(n)}\in\mathbb{R}^n: \forall t\geq0\,,\notag\\
    &\left(\sum_{i=1}^n e^{\lambda_i^{(n)} t}(v_i^{(n)})_1 v_i^{(n)}\right)^{\top}e^{(n)} \leq 0\bigg\},\label{eq:cone_Cn}
\end{align}
where $\lambda_i^{(n)}$ (resp. $v_i^{(n)}$), $i=1,\dots,n$, are the eigenvalues (resp. normalized eigenvectors) of $S^{(n)}A^{(n)}_{\gamma}S^{(n)}$. For all $\kappa < 0$, and every initial estimation error $e_0^{(n)} \in \mathcal{E}^{(n)}$, the state trajectory $x^{(n)}$ generated by the $(x^{(n)},e^{(n)})$-dynamics \eqref{eq:x_e_dyn_discr} with initial condition $(x_0^{(n)}\hspace{0.2cm}e_0^{(n)})^{\top}$ is nonnegative for all nonnegative initial states $x_0^{(n)}$, that is, for all $x_0^{(n)}\geq0$ and for all $t \geq 0$, $x^{(n)}(t)\geq 0$ .  \\
Moreover, if $\gamma<0$, then $\mathbb{R}_-^n\subsetneq \mathcal{E}^{(n)}$.
\end{lem}
\begin{pf}
Let $\kappa<0$ and $\gamma\in\mathbb{R}$. Thanks to the symmetry of the matrix $S^{(n)}A^{(n)}_{\gamma}S^{(n)}$, the vectors $v_i^{(n)}$, $i=1,...,n$, form an orthonormal basis of $\mathbb{R}^n$. So, the estimation error $e^{(n)}(t)$ of \eqref{eq:x_e_dyn_discr} is given for any $e_0^{(n)}\in\mathbb{R}^n$ by 
    \begin{equation}\label{eq:dec_mod_e^n}
        e^{(n)}(t) = \sum_{i=1}^n {v_i^{(n)}}^{\top}e_0^{(n)}e^{\lambda_i^{(n)} t} v_i^{(n)}.
    \end{equation}
    Now assume that $e_0^{(n)}\in \mathcal{E}^{(n)}$. Then, 
    \begin{equation*}
        \kappa {\mathbb{e}_1^{(n)}}^{\top} e^{(n)}(t) = \kappa(e^{(n)}(t))_1\geq0
    \end{equation*}
    holds for all $t\geq0$, since $\kappa<0$. Therefore, by applying \cite[Part I, Theorem 2]{Farina2000} to the closed-loop LTI system $(A^{(n)}_{\kappa},b^{(n)},0^{\top})$, where $A^{(n)}_{\kappa}$ is Metzler and $b^{(n)}\geq0$, and by fixing $u(t)=\kappa {\mathbb{e}_1^{(n)}}^{\top} e^{(n)}(t)$ for all~$t\geq0$, it follows that $x^{(n)}(t)\geq 0$ for all $x_0^{(n)}\geq0$ and $t \geq 0$. \\
    Now assume that $e_0^{(n)}\in\mathbb{R}_-^n$. Then, \begin{equation}\label{eq:error_traj_nonpositive}
        e^{(n)}(t) = e^{S^{(n)}A^{(n)}_{\gamma}S^{(n)}t}e_0^{(n)}\leq0
    \end{equation}
    for all $t\geq0$, since $S^{(n)}A^{(n)}_{\gamma}S^{(n)}$ is Metzler. Inequality \eqref{eq:error_traj_nonpositive} combined with \eqref{eq:dec_mod_e^n} ensures that $e_0^{(n)}\in \mathcal{E}^{(n)}$, and so $\mathbb{R}_-^n\subseteq \mathcal{E}^{(n)}$. 
    If $\gamma<0$, the fact that $\mathbb{R}_-^n\neq \mathcal{E}^{(n)}$ is an immediate consequence of \eqref{eq:-P subset -C} in Theorem \ref{theorem:error_inP_for_positivity}$\alpha$ below. 
    \qed
\end{pf}
For all $\kappa<0$ and $\gamma\in\mathbb{R}$, the cone $\mathcal{E}^{(n)}$ characterizes admissible initial estimation errors $e_0^{(n)}$ that preserve nonnegativity of the state trajectories $x^{(n)}$ of \eqref{eq:x_e_dyn_discr}. It has the advantage of not being restricted to $\mathbb{R}_-^n$ when $\gamma<0$. However, practically, it is neither easy nor meaningful to build the elements of $\mathcal{E}^{(n)}$ because it requires checking vector inequalities for all $t\geq0$. This motivates the following result.
\begin{thm}\label{theorem:error_inP_for_positivity}
$\alpha$) For all $\gamma < 0$, 
\begin{equation}\label{eq:-P subset -C}
    \mathcal{F}^{(n)} \subseteq \mathcal{E}^{(n)} \quad \text{and}\quad \mathcal{F}^{(n)}\setminus\mathbb{R}_-^n \neq\emptyset,
\end{equation}
where $\mathcal{E}^{(n)}$ is the cone given in \eqref{eq:cone_Cn} and $\mathcal{F}^{(n)}$ denotes the polyhedral cone 
\begin{subequations}\label{eq:cone_Pn}
    \begin{align}
        \mathcal{F}^{(n)} :=& \left\{e^{(n)}\in\mathbb{R}^n :-{P^{(n)}}^{\top}e^{(n)}\in\mathbb{R}^n_{+}\right\}\label{eq:cone_Pn_checking}\\
        =&\left\{-{\left({P^{(n)}}^{\top}\right)}^{-1}a\,:a\in\mathbb{R}^n_{+}\right\}, \label{eq:cone_Pn_implementation}
    \end{align}
\end{subequations}
where
\begin{equation*}
    P^{(n)}:= \begin{pmatrix}
        p_1^{(n)} & \cdots &  p_n^{(n)}
    \end{pmatrix}\in\mathbb{R}^{n\times n},
\end{equation*}
and, for all $i=1,\dots,n$,
\begin{equation*}\label{eq:def_pi}
    p_i^{(n)} :=\sum_{k=i}^n\left[\prod_{j=1}^{i-1}\left(1-\frac{\lambda_k^{(n)}}{\lambda_j^{(n)}}\right)\right](v_k^{(n)})_1v_k^{(n)},
\end{equation*}
and $\lambda_i^{(n)}$ (resp. $v_i^{(n)}$) are the eigenvalues (resp. normalized eigenvectors) of $S^{(n)}A^{(n)}_{\gamma}S^{(n)}$. Assume that the eigenvalues $\lambda_i^{(n)}$, $i=1,\dots,n$, are ordered increasingly, i.e. $\lambda_{i}<\lambda_{i+1}$, for all $i=1,\dots,n$.\\
$\beta$) For all $\kappa<0$ and $\gamma < 0$, and for every initial estimation error $e_0^{(n)} \in \mathcal{F}^{(n)}\cup\mathbb{R}_-^n$, the state trajectory $x^{(n)}$ generated by the $(x^{(n)},e^{(n)})$-dynamics \eqref{eq:x_e_dyn_discr} with initial condition $(x_0^{(n)}\hspace{0.2cm}e_0^{(n)})^{\top}$ is nonnegative for all nonnegative initial states $x_0^{(n)}$.
\end{thm}
\begin{pf}
    In view of Lemma \ref{lemma:conditions_positive_state}, it suffices to prove \eqref{eq:-P subset -C}. Let us fix $\gamma < 0$. The vectors $v_i^{(n)}$, $i=1,...,n$, forming an orthonormal basis of $\mathbb{R}^n$, there exists $c\in\mathbb{R}\setminus\{0\}$ and some $j$ such that $cv_j^{(n)}$ does not belong to $\mathbb{R}_-^n$ and $c(v_j^{(n)})_1\leq0$. The latter inequality combined with the fact that $\lambda_i^{(n)}$, $i=1,\dots,n$, are negative and ordered increasingly, implies that $cv_j^{(n)}\in \mathcal{F}^{(n)}$. Therefore, $\mathcal{F}^{(n)}\setminus\mathbb{R}_-^n \neq\emptyset$. 
    Now, let us set $\mathcal{C}^{(n)} :=\left\{\sum_{i=1}^n e^{\lambda_i^{(n)} t}(v_i^{(n)})_1 v_i^{(n)}:t\geq0\right\},$ and $\mathcal{P}^{(n)} := \left\{P^{(n)}a:a\in\mathbb{R}^n_+\right\}.$
    Observe that $\mathcal{E}^{(n)}=-{\mathcal{C}^{(n)}}^*$ and $\mathcal{F}^{(n)}=-{\mathcal{P}^{(n)}}^*$ correspond to the symmetric reflections of the dual cones of $\mathcal{C}^{(n)}$ and $\mathcal{P}^{(n)}$, respectively  (see \cite[Part I, Section 3, p.16]{Rockafellar1970}) and \cite[Chapter 1, Section 2, p.5]{Berman1989}).
    So, by \cite[Chapter 1, Exercise (2.13)]{Berman1989}, it suffices to show that $\mathcal{C}^{(n)}\subseteq\mathcal{P}^{(n)}$. For this purpose, let $t\geq0$ and set $a^{(t)}:=\begin{pmatrix}
        a^{(t)}_1,\dots,a^{(t)}_n
    \end{pmatrix}\in\mathbb{R}^n$, where for all $i=1,\dots,n$ 
    \begin{equation*}
        a^{(t)}_i := \Delta_i^{(t)}\displaystyle\left(\prod_{k=1}^{i-1}|\lambda_k^{(n)}|\right)\left(\displaystyle\prod_{1\leq l <j\leq i}(\lambda_j^{(n)}-\lambda_l^{(n)})\right)^{-1},
    \end{equation*}
    and
    \begin{equation*}
        \Delta_i^{(t)}:=\sum_{k=1}^i(-1)^{i-k}e^{\lambda_kt}\prod_{\substack{1\leq l <j\leq i\\j,l\neq k}}(\lambda_j^{(n)}-\lambda_l^{(n)}).
    \end{equation*}
    A straightforward computation yields:
    \begin{equation}\label{eq:P^n a^t}
        P^{(n)}a^{(t)} = \sum_{i=1}^n e^{\lambda_i^{(n)} t}(v_i^{(n)})_1 v_i^{(n)}.
    \end{equation}
    Moreover, it can be shown that $\Delta_i^{(t)}$ is the determinant of an ''augmented'' Vandermonde matrix: 
    \begin{equation*}
        \Delta_i^{(t)} = \begin{vmatrix}
            1&\lambda_1^{(n)}&\cdots&{\lambda_1^{(n)}}^{i-2}&e^{\lambda_1^{(n)} t}\\
            1&\lambda_2^{(n)}&\cdots&{\lambda_2^{(n)}}^{i-2}&e^{\lambda_2^{(n)} t}\\
            \vdots&\vdots&&\vdots&\vdots\\
            1&\lambda_i^{(n)}&\cdots&{\lambda_i^{(n)}}^{i-2}&e^{\lambda_i^{(n)} t}
        \end{vmatrix},
    \end{equation*}
    and so $\Delta_i^{(t)}\geq0$ because $\lambda_i^{(n)}$, $i=1,\dots,n$, are ordered increasingly.
    Consequently, we find that $a^{(t)}_i\geq0$ for all $i=1,\dots,n$. Finally, invoking \eqref{eq:P^n a^t} with the nonnegativity of $a^{(t)}$ shows that $\sum_{i=1}^n e^{\lambda_i^{(n)} t}(v_i^{(n)})_1 v_i^{(n)}\in\mathcal{P}^{(n)},$
    and since $t\geq0$ is arbitrary, we have that~$\mathcal{C}^{(n)}\subseteq\mathcal{P}^{(n)}$. \qed
\end{pf}
\begin{rem}
$\alpha$) For $n=2$, there holds $\mathcal{F}^{(2)} = \mathcal{E}^{(2)}$. Indeed, this can be shown by taking $t=0$ and $t\rightarrow\infty$, respectively, in the inequality of \eqref{eq:cone_Cn} divided by $e^{\lambda_{F}^{(n)} t}$, where $\lambda_{F}^{(n)}$ is the largest (negative) eigenvalue of $S^{(n)}A^{(n)}_{\gamma}S^{(n)}$. In the general case, by means of \eqref{eq:cone_Pn_checking}, it is easy to check that there are elements of $\mathcal{F}^{(n)}$ that are not in $\mathbb{R}_-^n$.
In addition, \eqref{eq:cone_Pn_implementation} is easy to use for determining suitable initial estimation errors.\\
$\beta$) From a geometrical point of view, $\mathcal{F}^{(n)}$ is the polyhedral cone generated by the $n$ columns of $-{\left({P^{(n)}}^{\top}\right)}^{-1}$. By Theorem \ref{theorem:error_inP_for_positivity}, it contains, in particular, some vectors that are not in the nonpositive cone of $\mathbb{R}^n$.  
More specifically, the vectors $p_i^{(n)}$ are linear combinations of the eigenvectors of $S^{(n)}A^{(n)}_{\gamma}S^{(n)}$, where the coefficients depend on the eigenvalues. In particular, $p_n^{(n)} \in \mathrm{span}\{v_n^{(n)}\}\cap\mathbb{R}^n_+$, where the nonnegativity is deduced from Perron-Frobenius theorem (see e.g. \cite[Theorem 4]{Arrow1989}) and the fact that $\lambda_i^{(n)}$, $i=1,\dots,n$, are negative and ordered increasingly. Moreover, since $\{v_i^{(n)}\}$ is an orthonormal basis, $p_1^{(n)} = \sum_{k=1}^n(v_k^{(n)})_1v_k^{(n)} = \mathbb{e}_1^{(n)}$. This implies that the first line of $-{({P^{(n)}}^{\top})}^{-1}$ is ${-\mathbb{e}_1^{(n)}}^{\top}$. So, all the elements of $\mathcal{F}^{(n)}$ have a nonpositive first component.\\
$\gamma$) As a consequence of Lemma \ref{lemma:conditions_positive_state} and Theorem \ref{theorem:error_inP_for_positivity}, for every initial state $x_0^{(n)}\in\mathbb{R}_+^n$, picking any $\tilde{x}_0^{(n)}\in \mathcal{F}^{(n)}+x_0^{(n)}$ guarantees the nonnegativity of the state trajectory $x^{(n)}(\cdot)$. In the state estimation context, $x_0^{(n)}$ is unknown but this consequence implies that, if a region $R(x_0^{(n)})$ to which $x_0^{(n)}$ belongs is known, then $\mathcal{F}^{(n)}+R(x_0^{(n)})$ provides an appropriate set of suitable initial estimated states $\tilde{x}_0^{(n)}$. Moreover, for any initial estimated state $\tilde{x}_0^{(n)}$ in $\mathcal{F}^{(n)}$, the nonnegativity property of the state trajectory holds.
\end{rem}

Now we get back to the considered problem for the nominal PDE system.  
As for the discretized system, the final goal is to derive sufficient algebraic conditions for an initial estimation error $e_0$ to lead to a nonnegative state trajectory $x$, for any nonnegative initial state $x_0$. We will see later that the derived algebraic conditions depend also on the eigenvalues and the eigenfunctions of the generator of the error dynamics (see Theorem \ref{theorem:suitable_combili_eigenFunctions}). Recall that, ideally, these conditions should not be reduced to the nonpositivity property. While we naturally adopt a time-domain approach for the problem related to the finite-dimensional system, it becomes natural to use a frequency-domain approach (based on the resolvent operator linked to the $(x,e)$-dynamics) to tackle the infinite-dimensional problem; see, for instance, \cite[Corollary 11.4]{Batkai2017} or \cite[Proposition 2.2]{Laabissi2001}. In this context, Lemma~\ref{lemma:conditions_positive_state_PDE} introduces the cone $\mathcal{E}$, which is strictly larger than the nonpositive cone of $\mathrm{L}^2(0,L)$, of suitable initial estimation errors with respect to our problem. Note that Lemma \ref{lemma:conditions_positive_state_PDE} is instrumental for Theorem \ref{theorem:suitable_combili_eigenFunctions}$\beta$.
\begin{lem}\label{lemma:conditions_positive_state_PDE}
Let $\gamma\in\mathbb{R}$ and consider the cone \begin{equation}\label{eq:cone_D}
    \begin{aligned}
        \mathcal{E} := \{&e\in\mathrm{L}^2(0,L) : \exists\omega\geq0,\,
        \forall \lambda>\omega, \\&(R(\lambda,SA_{\gamma}S^{-1})^m e)(0)\leq 0,\, \forall m\in \mathbb{N}\setminus\{0\}\},
    \end{aligned}
\end{equation}
where $R(\lambda,SA_{\gamma}S^{-1})$ denotes the resolvent operator of $SA_{\gamma}S^{-1}$. For all $\kappa < 0$ and every initial estimation error $ e_0$ belonging to $\mathcal{E}$, the mild solution at any time $t\geq0$ of the $(x,e)$-dynamics \eqref{eq:system_x_e} with initial condition $(x_0\hspace{0.2cm}e_0)^{\top}$ belongs to $\mathrm{L}^2(0,L)_+ \times \mathrm{L}^2(0,L)$ for all nonnegative initial states $x_0$, that is 
\begin{equation*}
    (T_{\tilde{A}_{\kappa\gamma}}(t))\begin{pmatrix}
            x_0\\e_0
        \end{pmatrix}
\in \mathrm{L}^2(0,L)_+ \times \mathrm{L}^2(0,L), 
\end{equation*} 
for all $x_0\in\mathrm{L}^2(0,L)_+$ and $t \geq 0$, where $(T_{\tilde{A}_{\kappa\gamma}}(t))_{t\geq 0}$ is the $C_0$-semigroup generated by $\tilde{A}_{\kappa\gamma}$ on $\mathrm{L}^2(0,L) \times \mathrm{L}^2(0,L)$. \\
Moreover, if $\gamma<0$, then $\displaystyle (\mathrm{L}^2(0,L))_{-}\subsetneq\mathcal{E}$.
\end{lem}
\begin{pf}
    Let $\kappa<0$ and  $\gamma\in\mathbb{R}$. Note first that evaluating $R(\lambda,SA_{\gamma}S^{-1})^m e$ at $z=0$ in \eqref{eq:cone_D} makes sense because these functions are in $D(SA_{\gamma}S^{-1})\subseteq C(0,L)$.\\
    The inclusion $ (\mathrm{L}^2(0,L))_{-}\subseteq\mathcal{E}$ is satisfied because $SA_{\gamma}S^{-1}$ generates a positive $C_0$-semigroup on $\mathrm{L}^2(0,L)$; see Section \ref{section:estimation}. If $\gamma<0$, $(\mathrm{L}^2(0,L))_{-}\neq\mathcal{E}$ is an immediate consequence of \eqref{eq:-Q subset -D} in Theorem \ref{theorem:suitable_combili_eigenFunctions}$\alpha$ below.\\
    Now let $e_0\in\mathcal{E}$, $x_0\in\mathrm{L}^2(0,L)_+$, and $t\geq0$. Fix also $g := (g_1\hspace{0.2cm}g_2)^{\top}\in \mathrm{L}^2(0,L) \times \mathrm{L}^2(0,L)$ and $\lambda>0$. By using standard arguments, one gets:
    \begin{equation*}
        R(\lambda,\tilde{A}_{\kappa\gamma})g\\
        = \begin{pmatrix}
                R(\lambda,A_{\kappa})g_1+  f_1 \\
                R(\lambda,SA_{\gamma}S^{-1})g_2
            \end{pmatrix},
    \end{equation*}
    where, for all $z\in[0,L]$,
    \begin{equation}\label{eq:f}
        f_1(z) := \cosh(\mu z)\frac{c_1}{c_2}-\kappa\left(\frac{c_1}{c_2}+(R_{\gamma}(\lambda)g_2)(0)\right)\frac{\sinh(\mu z)}{\mu},
    \end{equation}
    with $\mu := \sqrt{\lambda D_a^{-1}}$, $R_{\gamma}(\lambda) := R(\lambda,SA_{\gamma}S^{-1})$, and
    \begin{equation*}
    \left\{
    \begin{aligned}
        c_1 &:= \kappa (R_{\gamma}(\lambda)g_2)(0)\cosh(\mu L),\\
        c_2 &:= \mu\sinh(\mu L)-\kappa\cosh(\mu L).
    \end{aligned}
    \right.
    \end{equation*}
    Definition \eqref{eq:f} leads to the identity $f_1 = \kappa (R_{\gamma}(\lambda)g_2)(0) f_2$, where the function $f_2$ satisfies 
    \begin{equation}\label{eq:cosh/c_2}
        f_2(z) := \frac{\cosh(\mu(z-L))}{c_2}>0,\quad \forall z\in [0,L], 
    \end{equation}
    since $\kappa<0$ and $\mu L>0$. Now, for all $m\in\mathbb{N}\setminus\{0\}$,
    \begin{equation*}\label{eq:resolvent_m}
        \begin{aligned}
        R&(\lambda,\tilde{A}_{\kappa\gamma})^mg\\
        = &\begin{pmatrix}
                R(\lambda,A_{\kappa})^m g_1 +  \displaystyle\sum_{i=1}^m\kappa (R_{\gamma}(\lambda)^ig_2)(0)R(\lambda,A_{\kappa})^{m-i}f_2 \\
                R(\lambda,SA_{\gamma}S^{-1})^mg_2
            \end{pmatrix}.
        \end{aligned}
    \end{equation*}
    Now combining the fact that $x_0\in \mathrm{L}^2(0,L)_+$ with the positivity of the $C_0$-semigroup generated by $A_{\kappa}$ on $\mathrm{L}^2(0,L)$, and using \eqref{eq:cosh/c_2} together with the conditions $e_0\in \mathcal{E}$ and $\kappa<0$, we conclude that there exists $\omega\geq0$ such that, for all $\lambda>\omega$ and $m\in\mathbb{N}\setminus\{0\}$, 
    \begin{equation*}
    R\left(\lambda,\tilde{A}_{\kappa\gamma}\right)^m\begin{pmatrix}
            x_0\\e_0
        \end{pmatrix}\in \mathrm{L}^2(0,L)_+ \times \mathrm{L}^2(0,L).
    \end{equation*}
    Therefore, by the exponential formula for $C_0$-semigroups \cite[p. 33]{Pazy1983},
    \begin{equation}\label{eq:exponential_formula}
        (T_{\tilde{A}_{\kappa\gamma}}(t))
        \begin{pmatrix}
            x_0\\e_0
        \end{pmatrix} = \lim_{m\to +\infty}\left[\frac{m}{t}R\left(\frac{m}{t},\tilde{A}_{\kappa\gamma}\right)\right]^m\begin{pmatrix}
            x_0\\e_0
        \end{pmatrix},
    \end{equation}
    and since $\mathrm{L}^2(0,L)_+ \times \mathrm{L}^2(0,L)$ is closed, we conclude that the state trajectory given by \eqref{eq:exponential_formula} is in  $\mathrm{L}^2(0,L)_+ \times \mathrm{L}^2(0,L)$.
    \qed
\end{pf}
By examining its proof, observe that Lemma \ref{lemma:conditions_positive_state_PDE} is still valid if $\omega$ is fixed to zero in the definition of $\mathcal{E}$. However, the resulting set is included in $\mathcal{E}$, and so it is more restrictive. Furthermore, as for the discretized case, this first natural set $\mathcal{E}$ is hard to determine. This motivates the study of an auxiliary set $\mathcal{F}$, defined in the following theorem : for initial estimation errors that are linear combinations of eigenfunctions of $SA_{\gamma}S^{-1}$, simple algebraic conditions on the coefficients are given for the resulting linear combination to be in $\mathcal{E}$, and so to be suitable. 
\begin{thm}\label{theorem:suitable_combili_eigenFunctions}
$\alpha$) For all $\gamma < 0$, 
\begin{equation}\label{eq:-Q subset -D}
    \mathcal{F} \subseteq \mathcal{E} \quad \text{and}\quad \mathcal{F}\setminus (\mathrm{L}^2(0,L))_{-}\neq\emptyset,
\end{equation}
where the cones $\mathcal{E}$ and $\mathcal{F}$ are defined by \eqref{eq:cone_D} and
\begin{equation}\label{eq:cone_Q}
     \mathcal{F}  :=  \bigcup_{\underset{|I|<\infty}{I\subset\mathbb{N}\setminus\{0\}}}\left\{\sum_{i\in I}a_i\phi_i:(a_i)_{i\in I}\in \Gamma_I^1\cap \Gamma_I^2\right\},
\end{equation}
respectively, with for all $I\subseteq\mathbb{N}\setminus\{0\}, |I|<\infty$,
\begin{equation*}
\left\{
\begin{aligned}
    \Gamma_I^1 &:= \left\{a\in\mathbb{R}^{|I|} : a^{\top}(c_i)_{i\in I}\leq0\right\},\\
    \Gamma_I^2 &:= \left\{(a_i)_{i\in I}\in\mathbb{R}^{|I|} : \min_{i\in I:a_i<0}\lambda_i > \max_{i\in I:a_i>0}\lambda_i\right\},
\end{aligned}
\right.
\end{equation*}
and $\lambda_i$ (resp. $\phi_i$), $i\in\mathbb{N}\setminus\{0\}$, are the eigenvalues (resp. normalized eigenfunctions) of $SA_{\gamma}S^{-1}$, that is $\lambda_i=-D_a\mu_i^2$ and $\phi_i(z) = c_i\cos(\mu_i z)$, $z\in[0,L]$, where $\mu_i$ is the $i$-th solution of $-\frac{\gamma}{\mu_i} = \tan(\mu_iL)$, and where the real parameters $c_i$ are strictly positive.\\
$\beta$) For all $\kappa<0$ and $ \gamma < 0$, and every initial estimation error $e_0$ belonging to $\mathcal{F}\cup(\mathrm{L}^2(0,L))_-$, the mild solution at any time $t\geq0$ of the $(x,e)$-dynamics \eqref{eq:system_x_e} with initial condition $(x_0\hspace{0.2cm}e_0)^{\top}$ belongs to $\mathrm{L}^2(0,L)_+ \times \mathrm{L}^2(0,L)$ for all nonnegative initial states $x_0$. 
\end{thm}
\begin{pf}
    In view of Lemma \ref{lemma:conditions_positive_state_PDE}, it suffices to prove \eqref{eq:-Q subset -D}. Let us fix $\gamma<0$. It is clear that $c\phi_j\in\mathcal{F}$ for any real $c<0$ and $j\neq1$. However, such $c\phi_j\notin \mathrm{L}^2(0,L))_-$. This implies that $\mathcal{F}\setminus (\mathrm{L}^2(0,L))_{-}\neq\emptyset$. Now, let $I\subseteq\mathbb{N}\setminus\{0\}, |I|<\infty$, and $(a_i)_{i\in I} \in \Gamma_I^1\cap \Gamma_I^2$. Then, $e = \sum_{i\in I}a_i\phi_i\in \mathcal{F} $. Now, let $\lambda>0$ and $m\in\mathbb{N}\setminus\{0\}$. Applying \cite[Theorem 3.2.8]{Curtain_Zwart2020} to $SA_{\gamma}S^{-1}$ gives
    \begin{equation*}
        \begin{aligned}
            R(\lambda,SA_{\gamma}S^{-1}) e = \sum_{i\in I}a_i\frac{1}{\lambda-\lambda_i}\phi_i,
        \end{aligned}
    \end{equation*}
    and so, by induction, 
    \begin{equation}\label{eq:resolvent_m_combili}
        R(\lambda,SA_{\gamma}S^{-1})^m e = \sum_{i\in I}a_i\frac{1}{(\lambda-\lambda_i)^m}\phi_i.
    \end{equation}
    The negativity of $\lambda_i$, $i\in\mathbb{N}\setminus\{0\}$, combined with the positivity of $c_i$, $i\in\mathbb{N}\setminus\{0\}$, and routine estimates of \eqref{eq:resolvent_m_combili} at zero yield 
    \begin{equation}\label{eq:estimate_R^m(0)}
        \begin{aligned}
            &(R(\lambda,SA_{\gamma}S^{-1})^m e)(0)  \\&= \sum_{i\in I:a_i>0}a_i\frac{1}{(\lambda-\lambda_i)^m}c_i + \sum_{i\in I:a_i<0}a_i\frac{1}{(\lambda-\lambda_i)^m}c_i\\
            &\leq
            M_{m,\lambda}^+\sum_{i\in I:a_i>0}a_i c_i + M_{m,\lambda}^-\sum_{i\in I:a_i<0}a_ic_i,
        \end{aligned}
    \end{equation}
    where
    \begin{equation*}
        M_{m,\lambda}^+ := \max_{i\in I:a_i>0}\frac{1}{(\lambda-\lambda_i)^m} = \frac{1}{(\lambda-\max_{i\in I:a_i>0}\lambda_i)^m},
    \end{equation*}
    and
    \begin{equation*}
        M_{m,\lambda}^- := \min_{i\in I:a_i<0}\frac{1}{(\lambda-\lambda_i)^m} = \frac{1}{(\lambda-\min_{i\in I:a_i<0}\lambda_i)^m}.
    \end{equation*}
    Now observe that $(a_i)_{i\in I} \in \Gamma_I^1$ ($(a_i)_{i\in I} \in \Gamma_I^2$, resp.) ensures that
    \begin{equation}\label{eq:bounds_1}
        -\frac{\sum_{i\in I:a_i>0}a_i c_i}{\sum_{i\in I:a_i<0}a_i c_i}\leq 1 \quad \left(1 < \frac{M_{m,\lambda}^-}{M_{m,\lambda}^+} , \text{resp.} \right) 
    \end{equation}
    Combining inequalities \eqref{eq:estimate_R^m(0)} and \eqref{eq:bounds_1}, it follows that $(R(\lambda,SA_{\gamma}S^{-1})^m e)(0)\leq0$. Therefore, the property $e\in\mathcal{E}$ is satisfied for $\omega=0$.
    \qed
\end{pf}
\begin{rem}
$\alpha$) Analogously to Theorem \ref{theorem:error_inP_for_positivity} for the discretized case, Theorem \ref{theorem:suitable_combili_eigenFunctions} provides sufficient algebraic conditions for an initial estimation error $e_0$ to lead to a nonnegative state trajectory $x$, for any nonnegative initial state $x_0$. Here, the initial estimation errors in $\mathcal{F}$ are linear combinations of eigenfunctions of $SA_{\gamma}S^{-1}$. The derived conditions relate to the coefficients of these linear combinations and depend on the eigenvalues (via  $\Gamma_I^2$) and the eigenfunctions (via $\Gamma_I^1$) that are linked to the error dynamics, that is, to the operator $SA_{\gamma}S^{-1}$. More precisely, the set $\Gamma_I^2$ determines the signs of the coefficients $a_i$. Indeed, $(a_i)_{i\in I}\in \Gamma_I^2$ if, and only if, there exists $i^*\in I$ such that
\begin{equation}\label{eq:def_orthant_O_i*}
    \left\{
        \begin{aligned}
        &a_i\geq0\,, \forall i \in I \text{ s.t. } \lambda_i<\lambda_{i^*}\,,\\
        &a_{i^*}\in\mathbb{R}\,,\\
        &a_i\leq0\,, \forall i \in I \text{ s.t. } \lambda_{i^*}<\lambda_i\,.
        \end{aligned}
    \right.
\end{equation} 
Observe that the vectors $(a_i)_{i\in I}$ satisfying \eqref{eq:def_orthant_O_i*} form the union of two orthants of $\mathbb{R}^{|I|}$, that we denote $O_{i^*}^+$ and $O_{i^*}^-$, respectively. So, $\Gamma_I^2 = \cup_{i^*\in I} (O_{i^*}^+\cup O_{i^*}^-)$. Then, once the signs have been fixed, $\Gamma_I^1$ is more about the magnitudes of the coefficients. It actually defines a closed half-space in $\mathbb{R}^{|I|}$ related to the strictly positive vector $(c_i)_{i\in I}$.\\
$\beta$) In view of Theorem \ref{theorem:suitable_combili_eigenFunctions}, there are many other initial estimation errors than the (trivial) nonpositive functions of $\mathrm{L}^2(0,L)$, that are suitable. Some of these elements belong to $\mathcal{F}\setminus (\mathrm{L}^2(0,L))_{-}$. In contrast, using arguments similar to those in Theorem \ref{theorem:suitable_combili_eigenFunctions}, we can show that $\tilde{\mathcal{F}}\nsubseteq \mathcal{E}$, where 
\begin{equation*}
    \tilde{\mathcal{F}}  :=  \bigcup_{\underset{|I|<\infty}{I\subseteq\mathbb{N}\setminus\{0\}}}\left\{\sum_{i\in I}a_i\phi_i:(a_i)_{i\in I}\in (\Gamma_I^1)^c\cup \tilde{\Gamma}_I^2\right\}, 
\end{equation*}
with, for all $I\subseteq\mathbb{N}\setminus\{0\}, |I|<\infty$,
\begin{equation*}
    \tilde{\Gamma}_I^2 := \left\{(a_i)_{i\in I}\in\mathbb{R}^{|I|} : \min_{i\in I:a_i>0}\lambda_i > \max_{i\in I:a_i<0}\lambda_i\right\}.
\end{equation*}
$\gamma$) Because of the inequality in $\Gamma_I^1$ and/or the strict positivity of $(c_i)_{i\in I}$, two other observations can be made regarding any $e_0\in\mathcal{F}$: at least one coefficient related to $e_0$ is nonpositive, and $e_0(0)\leq 0$.
\end{rem}

Theorem \ref{theorem:suitable_combili_eigenFunctions} establishes a positivity property for mild solutions of the $(x,e)$-dynamics \eqref{eq:system_x_e}. In the following result, we now turn our attention to classical solutions of the same dynamics.
\begin{cor}\label{cor:classical_solutions}
    Let $\kappa\in\mathbb{R}$ and consider the linear functional $F$ 
    defined by $F(x) := \frac{-1}{\kappa} \frac{{\rm  d}  x}{{\rm  d}  z}(0) - x(0),$
    with domain $D(F) := \{x\in \mathrm{H}^2(0,L) :  \frac{{\rm  d}  x}{{\rm  d}  z}(L)=0 \}.$
    For all $\kappa, \gamma < 0$, and every nonnegative initial state $x_0$, there are initial estimation errors $e_0$ belonging to $\mathcal{F}$, defined by \eqref{eq:cone_Q}, such that the classical solution at any time $t\geq0$ of the $(x,e)$-dynamics \eqref{eq:system_x_e} with initial condition $(x_0\hspace{0.2cm}e_0)^{\top}$ belongs to $(\mathrm{L}^2(0,L)_+ \times \mathrm{L}^2(0,L))\cap D(\tilde{A}_{\kappa\gamma})$ if, and only if, $x_0$ belongs to $D(F)$ and $F(x_0)\leq0$.
\end{cor}
\begin{pf}
    Let $\kappa<0$, $\gamma < 0$ and $x_0\in \mathrm{L}^2(0,L)_+$. Since $\mathcal{F}$ is a cone included in $D(SA_{\gamma}S^{-1})$ and by Theorem \ref{theorem:suitable_combili_eigenFunctions}, the sufficiency of the condition follows by taking $e_0  = c\tilde{e}_0$, where $\tilde{e}_0\in\mathcal{F}$, $\tilde{e}_0(0)\neq0$, and $c := F(x_0)/\tilde{e}_0(0)$.
    The necessity is easily proved thanks to the definitions of $\mathcal{F}$ and $D(\tilde{A}_{\kappa\gamma})$. \qed
\end{pf}

\subsection{Numerical simulations}\label{subsection:numerical_simulations}
This subsection is devoted to numerical simulations in order to illustrate some results from Subsections \ref{subsection:Nonpositivity} and \ref{subsection:Cones_design}. Fix $\kappa<0$ and $\gamma<0$ and first consider
\begin{equation}\label{eq:CI}
    \begin{pmatrix}
        x_{0,a}\\\tilde{x}_{0,a}
    \end{pmatrix}
    := \sum_{i=1}^{N}\left\langle \bar{X},\Psi_i\right\rangle \Phi_i,
\end{equation}
where $\langle\cdot,\cdot\rangle$ is the usual inner product on $\mathrm{L}^2(0,L)\times \mathrm{L}^2(0,L)$, $N\geq1$, $\bar{X}$ is given by \eqref{eq:counter_example},~$\{\Phi_i,i\geq1\}$ are the~$\mathrm{L}^2(0,L)\times \mathrm{L}^2(0,L)$-normalized eigenfunctions of the Riesz spectral operator $A_{\kappa\gamma}$, and $\{\Psi_i,i\geq1\}$ is the associated biorthogonal sequence. The corresponding eigenvalues $\lambda_i$ are ordered decreasingly, i.e. $\lambda_{i+1}<\lambda_{i}$, for all $i\geq1$. By construction, $(x_{0,a}\hspace{0.2cm}\tilde{x}_{0,a})^{\top}$ defined in \eqref{eq:CI} belongs to $D(A_{\kappa\gamma})$. Moreover, \eqref{eq:CI} provides a $\mathrm{L}^2(0,L)\times \mathrm{L}^2(0,L)$-approximation of $\bar{X}$. Hence, since $\bar{X}$ is positive, 
the initial condition $(x_{0,a}\hspace{0.2cm}\tilde{x}_{0,a})^{\top}$ is nonnegative for sufficiently large $N$. Figure \ref{fig:counterexample} shows the state (Panel \ref{fig:counterexample_x}) and estimated state (Panel \ref{fig:counterexample_xtilde}) trajectories obtained from the nonnegative initial condition \eqref{eq:CI}, with the parameters $L=1, D_a=1, \kappa=-1, \gamma=-2$, and $N=20$.
\begin{figure}
    \centering
    \begin{subfigure}[]{0.5\textwidth}
        \centering
        \includegraphics[width = 0.85\textwidth]{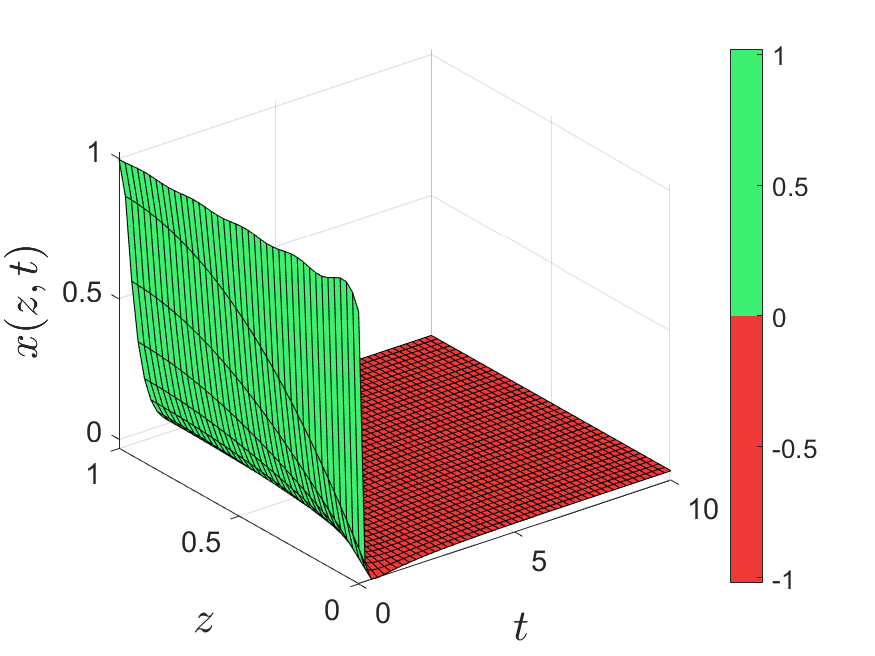}
        \caption{}
        \label{fig:counterexample_x}
    \end{subfigure}\\
    \begin{subfigure}[]{0.5\textwidth}
        \centering
        \includegraphics[width = 0.85\textwidth]{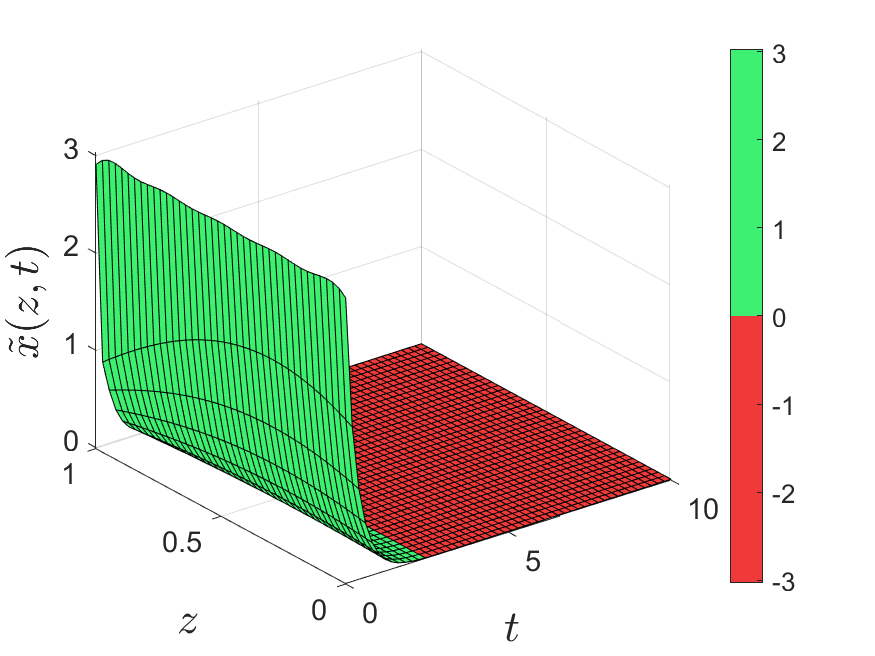}
        \caption{}
        \label{fig:counterexample_xtilde}
    \end{subfigure}%
    \caption{State (panel (a)) and estimated state (panel (b)) trajectories related to \eqref{eq:CI} ($L=1, D_a=1, \kappa=-1, \gamma=-2$, and $N=20$).}
    \label{fig:counterexample}
\end{figure}  
In this subsection, the figures have been obtained by using MATLAB R2024b and by discretizing the spatial domain $[0,L]$ into 40 points and using finite difference schemes. Observe that although both trajectories converge to zero, as expected from Proposition \ref{prop:stability_Atilde_kappagamma}, they become negative after some time. This illustrates Proposition \ref{prop:non_positivity_A_kappagamma}. \\
We now show that the initial estimation error $e_{0,a}:=\tilde{x}_{0,a}- x_{0,a}$
does actually not belong to $\mathcal{F}$ (and even not to $\mathcal{E}$) when $\kappa=-1$ and $\gamma-2$. To this end, let us fix $(\Phi_{i,1}\hspace{0.2cm}\Phi_{i,2})^{\top} := \Phi_i$ where $\Phi_{i,1},\Phi_{i,2}\in \mathrm{L}^2(0,L)$, and observe that 
\begin{equation}\label{eq:def_e0a}
    e_{0,a} = \sum_{i=1}^{N}\left\langle \bar{X},\Psi_i\right\rangle (\Phi_{i,2}-\Phi_{i,1}).
\end{equation}
A direct computation of $\Phi_{i,1}$ and $\Phi_{i,2}$ leads to $\Phi_{i,2}-\Phi_{i,1} = 0$, for all $i\in\mathbb{N}\setminus\{0\}$ such that $\lambda_i\in\sigma_p(A_{\kappa})$, and $\Phi_{i,2}-\Phi_{i,1} = \frac{1}{\|\Phi_i\|}(\frac{\gamma}{\kappa}-1)\cos(\sqrt{-\lambda_i/D_a} z)$, for all $i\in\mathbb{N}\setminus\{0\}$ such that $\lambda_i\in\sigma_p(SA_{\gamma}S^{-1})$. Since $\sigma_p(A_{\kappa\gamma}) = \sigma_p(A_{\kappa})\cup\sigma_p(SA_{\gamma}S^{-1})$, it follows from \eqref{eq:def_e0a} that $e_{0,a}=\sum_{i\in I_a}a_{i,a}\phi_i$, for some $I_a\subseteq\mathbb{N}\setminus\{0\}, |I_a|<\infty$, and $(a_{i,a})_{i\in I_a}\in \mathbb{R}^{|I_a|}$. When $\kappa=-1$ and $\gamma=-2$, as in Figure \ref{fig:counterexample}, a straightforward evaluation yields $e_{0,a}(0)\approx 2>0$. Consequently, $e_{0,a}\notin \mathcal{F}\cup(\mathrm{L}^2(0,L))_-$ but~$e_{0,a}\in \tilde{\mathcal{F}}$ and so $e_{0,a}\notin \mathcal{E}$.

As an alternative, take any $e_{0,b}\in \mathcal{F}$. By Theorem \ref{theorem:suitable_combili_eigenFunctions}, the mild solution of the $(x,e)$-dynamics \eqref{eq:system_x_e} associated to $(x_{0,a}\hspace{0.2cm}e_{0,b})^{\top}$ belongs to $\mathrm{L}^2(0,L)_+ \times \mathrm{L}^2(0,L)$ at any time. For the design, Algorithm \ref{alg:implementation_-Q*} is a simple way to implement any element of $\mathcal{F}$ for all $\gamma<0$. 
\begin{algorithm}
\caption{Implementation of $\mathcal{F}$}
\label{alg:implementation_-Q*}
\begin{algorithmic}[1]
\Require $D_a,L>0$ and $\gamma<0$
\Ensure $e_0\in \mathcal{F}$
\State Initialize $I\subseteq\mathbb{N}\setminus\{0\}, |I|<\infty$
\For{$i\in I$}
    \State Compute the $i$-th solution of $-\frac{\gamma}{\mu_i} = \tan(\mu_iL)$
    \State $\lambda_i \gets -D_a\mu_i^2$
    \State Initialize $c_i>0$ 
    \State $\phi_i \gets c_i\cos(\mu_i\,\cdot)$
\EndFor
\State $j\gets \mathrm{arg\,max}_{i \in I} \lambda_i$ and initialize $i^*\in I\setminus \{j\}$
\State Initialize $(a_i)_{i\in I\setminus\{j\}}\in\mathbb{R}^{|I|-1}$ such that
\begin{equation*}
    \left\{
        \begin{aligned}
        &a_i\geq0\,, \forall i \in I\setminus\{j\} \text{ s.t. } \lambda_i<\lambda_{i^*}\,,\\
        &a_{i^*}\in\mathbb{R}\,,\\
        &a_i\leq0\,, \forall i \in I\setminus\{j\} \text{ s.t. } \lambda_{i^*}<\lambda_i\,.
        \end{aligned}
    \right.
\end{equation*} 
\State Initialize $a_j\in\mathbb{R}$ such that $a_j\leq-\sum_{i\in I\setminus \{j\}}a_i\frac{c_i}{c_j}$ and $a_j\leq0$   \label{line:initialization_aj}
\State \Return $e_0 \gets \sum_{i\in I}a_i\phi_i$
\end{algorithmic}
\end{algorithm}
However, note that, when $\kappa=-1$ and $\gamma=-2$, $F(x_{0,a}) = e_{0,a}(0) \approx 2 > 0$ while $x_{0,a}\in D(F)$, since $(x_{0,a}\hspace{0.2cm}e_{0,a})^{\top}\in D(\tilde{A}_{-1,-2})$. 
Therefore, by Corollary \ref{cor:classical_solutions} combined with Theorem \ref{theorem:suitable_combili_eigenFunctions}, the mild solution of the $(x,e)$-dynamics \eqref{eq:system_x_e} associated to $(x_{0,a}\hspace{0.2cm}e_{0,b})^{\top}$ does not belong to $D(\tilde{A}_{\kappa\gamma})$ at some times.

Consider now the nonnegative initial state $x_{0,b}$ defined by $x_{0,b}(z) := -\kappa$ ($\approx x_{0,a}(z)$), $z\in[0,L]$, with $\kappa<0$. In that case, it is clear that $x_{0,b}\in D(F)$ and $F(x_{0,b}) = \kappa<0$.
By Corollary \ref{cor:classical_solutions}, for all $\gamma<0$, we can therefore design $e_{0}\in\mathcal{F}$ such that the classical solution $x$ of the $(x,e)$-dynamics related to such $(x_{0,b}\hspace{0.2cm}e_{0})^{\top}$ is well-defined and nonnegative, that is $x(\cdot,t)$ is nonnegative for all $t\geq0$. This construction is the subject of Algorithm \ref{alg:implementation_-Q*_domain}.
\begin{algorithm}
    \caption{Implementation of $e_0\in\mathcal{F}$ satisfying the boundary condition}
    \label{alg:implementation_-Q*_domain}
        \begin{algorithmic}[1]
        \Require $D_a,L>0$, $\kappa\in\mathbb{R}$, $\gamma<0$ and $x_0\in D_L$ such that $F(x_{0})\leq0$
        \Ensure $e_0\in \mathcal{F}$ and $(x_0\hspace{0.2cm}e_0)^{\top}\in D(\tilde{A}_{\kappa,\gamma})$
        \State $\tilde{e}_0\gets$Algorithm \ref{alg:implementation_-Q*}bis
        \Statex \Comment{Same as Algorithm \ref{alg:implementation_-Q*}, except that in line~\ref{line:initialization_aj} the inequality is replaced by $a_j<-\sum_{i\in I\setminus\{j\}}a_i\frac{c_i}{c_j}$}
        \State $c\gets\frac{F(x_0)}{\tilde{e}_0(0)}$
        \State\Return $e_0\gets c\tilde{e}_0$ 
        \end{algorithmic}
\end{algorithm}
The nonnegativity of the classical solution $x$ is illustrated in Figure \ref{fig:example_x} with the parameters $L=1, D_a=1, \kappa=-1, \gamma=-2$ and the initial estimation error $e_{0,c}\in \mathcal{F}$ corresponding to $I = \{1,2,3,4,5\}$, $c_i=1$, $i\in I$, and coefficients $(a_i)_{i\in I}$ given by the following table.  
\begin{center}
\centering
\begin{tabular}{c|ccccc}
$i$ & 1 & 2 & 3 & 4 & 5 \\
\hline
$a_i$ & -0.0760 & -0.6221 & -0.3510 & -0.1964 & 0.5132
\end{tabular}
\end{center}
Figure \ref{fig:example_e0} shows this initial estimation error, which has been designed using Algorithm \ref{alg:implementation_-Q*_domain}. Observe, in particular, that $e_{0,c}$ is neither nonpositive nor nonnegative. Finally, Figure \ref{fig:example_xtilde} depicts the estimated state trajectory associated with the $(x,e)$-dynamics and $(x_{0,b}\hspace{0.2cm}e_{0,c})^{\top}$. Here, the estimated state is not nonnegative at the beginning (i.e., at $t=0$), but it becomes nonnegative and remains so after a short transient period.
\begin{figure}[!b]
    \centering
    \begin{subfigure}[]{0.5\textwidth}
        \centering
        \includegraphics[width = 0.85\textwidth]{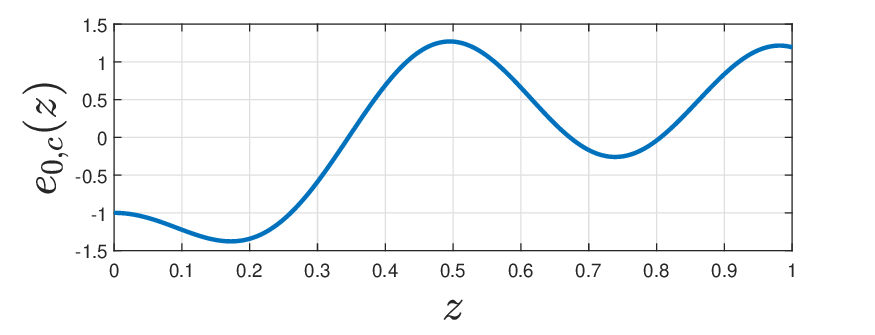}
        \caption{}
        \label{fig:example_e0}
    \end{subfigure}\\
    \begin{subfigure}[]{0.5\textwidth}
        \centering
        \includegraphics[width = 0.85\textwidth]{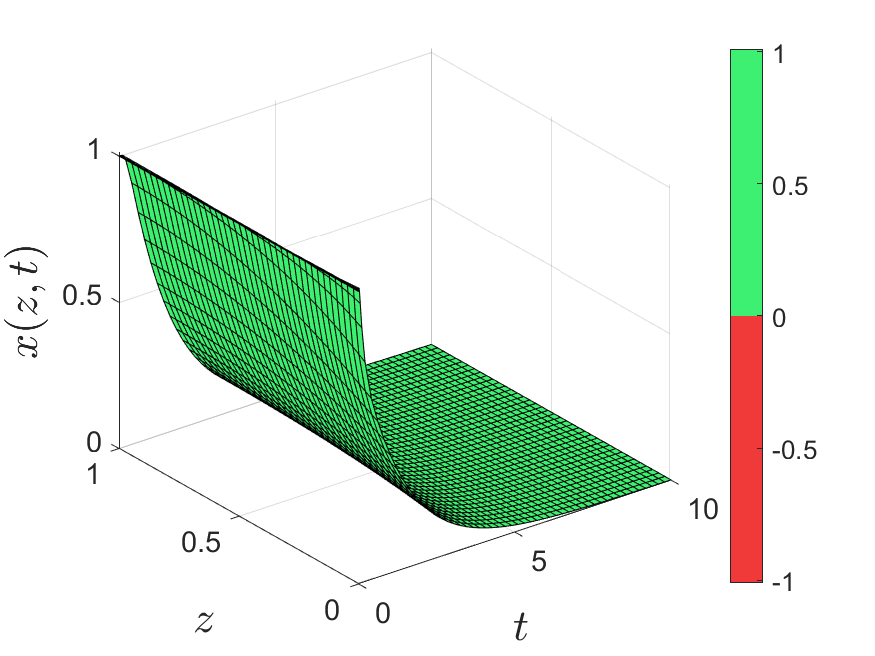}
        \caption{}
        \label{fig:example_x}
    \end{subfigure}\\
    \begin{subfigure}[]{0.5\textwidth}
        \centering
        \includegraphics[width = 0.85\textwidth]{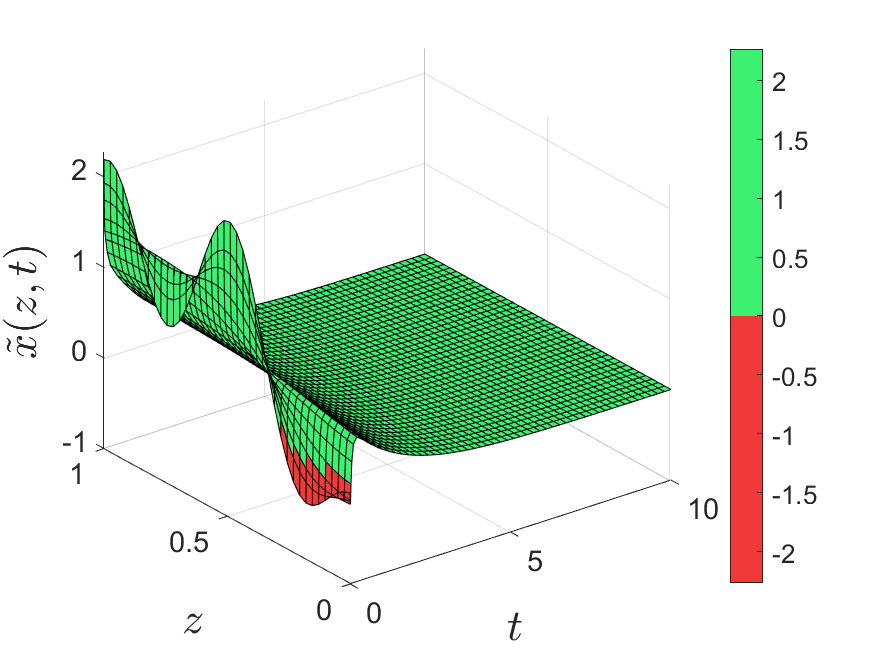}
        \caption{}
        \label{fig:example_xtilde}
    \end{subfigure}%
    \caption{State (panel (b)) and estimated state (panel (c)) trajectories related to the initial state $x_{0,b}$ and the initial estimation error $e_{0,c}$ (panel (a)) ($L=1, D_a=1, \kappa=-1, \gamma=-2$).}
    \label{fig:example}
\end{figure} 
To derive a nonnegative initial estimated state for $x_{0,b}$, it actually suffices to design $e_0$ like in Algorithm \ref{alg:implementation_-Q*_domain} with $e_0(z)\geq\kappa$, for all $z\in[0,L]$.

\section{Conclusion}\label{section:conclusion}
This paper dealt with a comprehensive analytical study of positive stabilization, positive state estimation, and observer-based positive stabilization for a one-dimensional boundary control diffusion system with point observation and control, together with its finite-dimensional spatial discretizations. Owing to the analytical tractability of the considered model, analytical and computable solutions were obtained for each design problem in both the continuous and discretized settings. In particular, optimal positively stabilizing state feedback laws and optimal positive observers were characterized through explicit feasibility conditions and closed-form optimal solutions. The two optimal design problems were interpreted as trade-offs between exploiting the desirable properties of low-gain state feedbacks and high-gain observers, respectively, while mitigating their inherent disadvantages and ensuring nonnegativity of the trajectories of interest. Furthermore, the observer-based positive stabilization problem was thoroughly investigated. It was shown that no considered observer-based state feedback associated with a positive observer can positively stabilize the considered system. To overcome this limitation, an alternative design strategy was proposed by identifying non trivial sets of initial estimation errors that guarantee the positivity of the real state trajectory while preserving implementability. The effectiveness of this proposed design strategy was demonstrated by numerical simulations. Throughout the paper, links between the studied system and its discretized version have been highlighted via convergence analysis and parallel arguments or results.

As an outlook for future research, the links established between the considered diffusion system and its approximation, the completeness of the overall study and the implementability of the solutions motivate an extension of the proposed study and methodology to more general classes of infinite-dimensional positive systems, including convection-diffusion-reaction systems subject to more general boundary conditions.

\begin{ack}                               
This research was conducted with the financial support of F.R.S-FNRS. Violaine Piengeon is a FNRS Research Fellow under the grant FC 58057. The scientific responsibility rests with its authors. The authors thank Anthony Hastir (UNamur, Belgium) for reviewing the manuscript and for insighful and helpul discussions.
\end{ack}

\appendix
\textbf{Appendix}  
\section{ Proof of Lemma \ref{lemma:feasability_practicality}: additional details:}
$\lambda$ is an eigenvalue of the matrix $A_{\kappa}^{(n)}$ if, and only if, $\lambda$ is a real zero of the function $g_{\kappa}$, or equivalently 
\begin{equation*}
        \kappa -\frac{p_2+\lambda}{p_1}-\frac{p_2^2}{p_1}\frac{[A^{(n)}-\lambda I]_{|n-2}}{[A^{(n)}-\lambda I]_{|n-1}} = 0 \; .
    \end{equation*}
Indeed, by applying the cofactor formula 
with respect to the first column of $A^{(n)}_{\kappa}-\lambda I$,  
one gets that $\lambda$ is an eigenvalue of $A^{(n)}_{\kappa}$ if, and only if, 
    \begin{equation}\label{eq:det_equal_0}
        (-p_2+p_1 \kappa -\lambda) [A^{(n)}-\lambda I]_{|n-1} -p_2^2 [A^{(n)}-\lambda I]_{|n-2} = 0 \; . 
   \end{equation}
Moreover, if \eqref{eq:det_equal_0} is satisfied, then $[A^{(n)}-\lambda I]_{|n-1} \neq 0$. Indeed, if this is not the case, then it follows from \eqref{eq:det_equal_0} that $[A^{(n)}-\lambda I]_{|n-2} = 0$. Since 
\begin{equation}\label{eq:recurrence_detF}
    [A^{(n)}-\lambda I]_{|i} = (-2p_2-\lambda) [A^{(n)}-\lambda I]_{|i-1}-p_2^2 [A^{(n)}-\lambda I]_{|i-2},
\end{equation}
for $i= 3, \cdots, n-1$.
it follows that $[A^{(n)}-\lambda I]_{|i}=0$ for $i=1,\hdots,n-1$, the case $i=1$ being clearly a contradiction.
\section{Proof of Theorem \ref{theorem:design_k1_opti}: additional details:}
The identity $[A^{(n)}]_{|n-2}{([A^{(n)}]_{|n-1})}^{-1} = -p_2^{-1}$ holds.
Indeed, by using \eqref{eq:recurrence_detF}, it can be easily shown by induction that, $[A^{(n)}]_{|i-1}{([A^{(n)}]_{|i})}^{-1} = -p_2^{-1}$, for $i=2, \cdots, n-1$.
\section{Proof of Lemma \ref{lemma:conv_sol}: additional details:}
        The sequence $(|u_n-v_n|)$ converges towards zero. Indeed, using the triangle inequality, \eqref{eq:p1 p2}, and the definition of binomial coefficients and factorials successively, for all $n\in\mathbb{N}\setminus\{0,1\}$,
        \begin{equation}\label{eq:bounds_un_vn}
            0\leq|u_n-v_n|\leq  \sum_{l=0}^{\infty} f_n(l),
        \end{equation}
        where $(f_n)_{n\in\mathbb{N}\setminus\{0,1\}}$ is the sequence of real-valued measurable functions given, for all $l\in\mathbb{N}$, by
        \begin{equation}\label{eq:def_fn(l)}
            f_n(l) := \left|1-g_n(l)\right|\frac{L^{2l+1}}{D_a^{l+\frac{1}{2}}}\frac{\alpha^{l+1}}{(2l+1)!}\mathbb{1}_{\{0,1,...,n-1\}}(l),
        \end{equation}
        with $g_n(l):= \frac{1}{(n-1)^{2l+1}}{\prod_{i=1}^{2l+1}(n-l-1+i)}$. Observe that the fundamental theorem on symmetric polynomials and the fact that $\sigma_0$ is equal to 1 yield that, for all $n\in\mathbb{N}\setminus\{0,1\}$, $l\in\{0,1,...,n-1\}$,
        \begin{align}
            g_n(l) 
            &= \frac{1}{(n-1)^{2l+1}}{\displaystyle\sum_{i=0}^{2l+1}\sigma_i\left(\left[-l+j\right]_{j=1}^{2l+1}\right)(n-1)^{2l+1-i}}\notag\\
            &=1+{\displaystyle\sum_{i=1}^{2l+1}\sigma_i\left(\left[-l+j\right]_{j=1}^{2l+1}\right)\frac{1}{(n-1)^{i}}}\,.\label{eq:def_gn(l)}
        \end{align}
        It is now easy to see that, for all $l\in\mathbb{N}$, 
        \begin{equation}\label{eq:conv_fn}
            f_n(l)\rightarrow 0 \quad \text{as}\quad n\rightarrow\infty.     
        \end{equation}
        Moreover, from the definition of elementary symmetric polynomials, it follows that, for all $l\in\mathbb{N}\setminus\{0\}$, $i\in\{1,2,...,2l+1\}$,
        \begin{align*}
            |\sigma_i\left(\left[-l+j\right]_{j=1}^{2l+1}\right)|&\leq\sum_{1\leq j_1<...<j_i\leq 2l+1}\prod_{k=1}^{i}|-l+j_k|\\
            &\leq \sum_{1\leq j_1<...<j_i\leq 2l+1}2 l^i\\
            &= 2 l^i \binom{2l+1}{i}.
        \end{align*}     
        Since $\frac{l}{n-1}\mathbb{1}_{\{0,1,...,n-1\}}(l)\leq1$ for all $n\in\mathbb{N}\setminus\{0,1\},\,l\in \mathbb{N}$, the following estimate 
        \begin{equation}\label{eq:bound_fn}
            f_n(l) \leq  2(2^{2l+1}-1)\frac{L^{2l+1}}{D_a^{l+\frac{1}{2}}}\frac{\alpha^{l+1}}{(2l+1)!} =: f(l),
        \end{equation}
        can be derived using the triangle inequality in \eqref{eq:def_fn(l)} with \eqref{eq:def_gn(l)}, and the binomial theorem for $2^{2l+1}$. By arguments based on d'Alembert's ratio tests, it is easy to show that
        \begin{equation}\label{eq:series_convergent}
            \sum_{l=0}^{\infty}f(l)<+\infty.
        \end{equation}
        From the Lebesgue's dominated convergence theorem applied to the sequence $(f_n)_{n\in\mathbb{N}\setminus\{0,1\}}$ where \eqref{eq:conv_fn}, \eqref{eq:bound_fn}, \eqref{eq:series_convergent} are satisfied, it follows that~$\lim_{n\rightarrow\infty}\sum_{l=0}^{\infty} f_n(l) = 0$. The conclusion follows by using \eqref{eq:bounds_un_vn}.

\bibliographystyle{plain}        
\bibliography{autosam}           

@book{Curtain_Zwart2020,
  title     = {Introduction to Infinite-Dimensional Systems Theory. A State-Space Approach},
  author    = {R. Curtain and H. Zwart},
  series    = {Texts in Applied Mathematics},
  publisher = {Springer New York, NY},
  year      = {2020},
  url       = {https://doi.org/10.1007/978-1-0716-0590-5},
  Language  = {English}
}

@article{Dehaye2016,
    title = {Parameterization of positively stabilizing feedbacks for single-input positive systems},
    journal = {Systems \& Control Letters},
    volume = {98},
    pages = {57-64},
    year = {2016},
    issn = {0167-6911},
    doi = {https://doi.org/10.1016/j.sysconle.2016.10.010},
    url = {https://www.sciencedirect.com/science/article/pii/S0167691116301657},
    author = {J. Dehaye and J.J. Winkin},
}

@book{VandeWouwer2014,
  author       = {A. Vande Wouwer and P. Saucez and C. Vilas},
  title        = {Simulation of ODE/PDE Models with MATLAB®, OCTAVE and SCILAB. Scientific and Engineering Applications},
  year         = {2014},
  publisher    = {Springer Cham},
  address      = {Cham, Switzerland},
  url          = {https://doi.org/10.1007/978-3-319-06790-2},         
  Language     = {English}
}

@book{Crank1975,
  title        ={The Mathematics of Diffusion},
  author       ={J. Crank},
  isbn         ={9780198534112},
  lccn         ={75318921},
  series       ={Oxford science publications},
  url          ={https://books.google.be/books?id=eHANhZwVouYC},
  year         ={1975},
  publisher    ={Clarendon Press},
  Language     = {English}
}

@article{Laabissi2001,
title = {Trajectory analysis of nonisothermal tubular reactor nonlinear models},
journal = {Systems \& Control Letters},
volume = {42},
number = {3},
pages = {169-184},
year = {2001},
issn = {0167-6911},
doi = {https://doi.org/10.1016/S0167-6911(00)00088-8},
url = {https://www.sciencedirect.com/science/article/pii/S0167691100000888},
author = {M. Laabissi and M.E. Achhab and J.J. Winkin and D. Dochain},
}

@article{Noschese2013,
    author = {S. Noschese and L. Pasquini and L. Reichel},
    title = {Tridiagonal Toeplitz matrices: properties and novel applications},
    journal = {Numerical Linear Algebra with Applications},
    volume = {20},
    number = {2},
    pages = {302-326},
    year = {2013}
}

@article{Rami2011,
author = {M.A. Rami and F. Tadeo and U. Helmke},
title = {Positive observers for linear positive systems, and their implications},
journal = {International Journal of Control},
volume = {84},
number = {4},
pages = {716-725},
year = {2011},
publisher = {Taylor \& Francis},
}

@book{Batkai2017,
  author    = {A. B{\'a}tkai and M. Kramar Fijav{\v{z}} and Abdelaziz Rhandi},
  title     = {Positive Operator Semigroups: From Finite to Infinite Dimensions},
  series    = {Operator Theory: Advances and Applications},
  year      = {2017},
  publisher = {Birkh{\"a}user Cham},
  doi       = {10.1007/978-3-319-42813-0},
  edition   = {1},
  pages     = {xviii+364}
}

@article{Hastir2023,
title = {A frequency domain approach to Kalman filtering on Hilbert spaces: Application to Sturm–Liouville systems with pointwise measurement},
journal = {Annual Reviews in Control},
volume = {55},
pages = {379-389},
year = {2023},
issn = {1367-5788},
doi = {https://doi.org/10.1016/j.arcontrol.2023.02.003},
author = {A. Hastir and J. Mohet and J.J. Winkin}
}

@book {Farina2000,
    AUTHOR = {L. Farina and S. Rinaldi},
     TITLE = {Positive linear systems: Theory and applications},
 PUBLISHER = {Wiley-Interscience, New York},
      YEAR = {2000},
     PAGES = {x+305},
      ISBN = {0-471-38456-9},
   MRCLASS = {93-02 (93A30 93B25 93C05)},
  MRNUMBER = {1784150 (2001g:93001)},
MRREVIEWER = {T. Kaczorek},
       DOI = {10.1002/9781118033029},
       URL = {http://dx.doi.org/10.1002/9781118033029},
}

@book{Berman1989,
  title={Nonnegative Matrices in Dynamic Systems},
  author={A. Berman and M. Neumann and R.J. Stern},
  isbn={9780471620747},
  lccn={88339344},
  series={Pure and Applied Mathematics: A Wiley Series of Texts, Monographs and Tracts},
  url={https://books.google.be/books?id=VQrvAAAAMAAJ},
  year={1989},
  publisher={Wiley}
}

@article{Lin2009LowGA,
  title={Low gain and low-and-high gain feedback: A review and some recent results},
  author={Z. Lin},
  journal={2009 Chinese Control and Decision Conference},
  year={2009},
  pages={lii-lxi},
  url={https://api.semanticscholar.org/CorpusID:36717728}
}

@book{Rockafellar1970,
  title={Convex Analysis},
  author={R.T. Rockafellar},
  isbn={9780691080697},
  lccn={68056318},
  series={Princeton landmarks in mathematics and physics},
  url={https://books.google.be/books?id=OI4Ph2dXXhsC},
  year={1970},
  publisher={Princeton University Press}
}

@INPROCEEDINGS{Khalil2008,
  author={H.K. Khalil},
  booktitle={2008 International Conference on Control, Automation and Systems}, 
  title={High-gain observers in nonlinear feedback control}, 
  year={2008},
  volume={},
  number={},
  pages={xlvii-lvii},
}

@article{Kitsos2021,
title = {High-gain observer for 3$\times$3 linear heterodirectional hyperbolic systems},
journal = {Automatica},
volume = {129},
pages = {109607},
year = {2021},
issn = {0005-1098},
author = {C. Kitsos and G. Besançon and C. Prieur},
}

@inproceedings{PiengeonWinkin2026,
    author = {V. Piengeon and J.J. Winkin},
    title = {Design of an optimal positive state observer
for a diffusion system with point observation},
    booktitle = {Proceedings of the 27th MTNS symposium},
    month = {17-21 August},
    year = {2026},
    adress = {Cambridge, UK}
}

@book{Pazy1983,
  author    = {A. Pazy},
  title     = {Semigroups of Linear Operators and Applications to Partial Differential Equations},
  publisher = {Springer},
  address   = {Berlin},
  year      = {1983}
}

@article{Binid2021,
    author = {A. Binid and M.E. Achhab and M. Laabissi},
    title = {Positive observers for linear positive systems in a Hilbert lattice space},
    journal = {IMA Journal of Mathematical Control and Information},
    volume = {38},
    number = {1},
    pages = {143-158},
    year = {2021},
    month = {03},
}

@techreport{Arrow1989,
  author      = {K.J. Arrow},
  title       = {A ``Dynamic'' Proof of the Frobenius-Perron Theorem for Metzler Matrices},
  institution = {Institute for Mathematical Studies in the Social Sciences, Stanford University},
  series      = {Economics Series},
  number      = {542},
  address     = {Stanford, California},
  year        = {1989}
}

@book{Haddad2010,
  title={Nonnegative and compartmental dynamical systems},
  author={W.M. Haddad and V. Chellaboina and Q. Hui},
  year={2010},
  publisher={Princeton University Press}
}

@book{Berman1994,
  title={Nonnegative matrices in the mathematical sciences},
  author={A. Berman and R.J. Plemmons},
  year={1994},
  publisher={SIAM}
}

@book{Arendt1986,
  title={One-parameter semigroups of positive operators},
  author={W. Arendt and A. Grabosch and G. Greiner and U. Moustakas and R. Nagel and U. Schlotterbeck and U. Groh and H.P. Lotz and F. Neubrander},
  volume={1184},
  year={1986},
  publisher={Springer}
}

@book{Smith1995,
  title={Monotone dynamical systems: an introduction to the theory of competitive and cooperative systems: an introduction to the theory of competitive and cooperative systems},
  author={H.L. Smith},
  number={41},
  year={1995},
  publisher={American Mathematical Soc.}
}

@article{Rantzer2021,
  title={Scalable control of positive systems},
  author={A. Rantzer and M.E. Valcher},
  journal={Annual Review of Control, Robotics, and Autonomous Systems},
  volume={4},
  number={1},
  pages={319--341},
  year={2021},
  publisher={Annual Reviews}
}

@INPROCEEDINGS{Abouzaid2010,
  author={B. Abouzaid and J.J. Winkin and V. Wertz},
  booktitle={49th IEEE Conference on Decision and Control (CDC)}, 
  title={Positive stabilization of infinite-dimensional linear systems}, 
  year={2010},
  volume={},
  number={},
  pages={845-850},
  doi={10.1109/CDC.2010.5717896}
  }

@article{Yupanqui2021,
  title={A concise review of state estimation techniques for partial differential equation systems},
  author={I.F. Yupanqui Tello and A. Vande Wouwer and D. Coutinho},
  journal={Mathematics},
  volume={9},
  number={24},
  pages={3180},
  year={2021},
  publisher={MDPI}
}

@book{Sontag2013,
  title={Mathematical control theory: deterministic finite dimensional systems},
  author={E.D. Sontag},
  volume={6},
  year={2013},
  publisher={Springer Science \& Business Media}
}

@InProceedings{Laabissi2003,
    author="M. Laabissi and M.E. Achhab and J.J. Winkin and D. Dochain",
    editor="L. Benvenuti and A. De Santis and L. Farina",
    title="Positivity and Invariance Properties of Nonisothermal Tubular Reactor Nonlinear Models",
    booktitle="Positive Systems",
    year="2003",
    publisher="Springer Berlin Heidelberg",
    address="Berlin, Heidelberg",
    pages="159--166"
}

@inproceedings{Achhab2017,
    author="M.E. Achhab and J.J. Winkin",
    editor="F. Cacace and L. Farina and R. Setola and A. Germani",
    title="Positive Stabilization of a Class of Infinite-Dimensional Positive Systems",
    bookTitle="Positive Systems : Theory and Applications (POSTA 2016) Rome, Italy, September 14-16, 2016",
    year="2017",
    publisher="Springer International Publishing",
    address="Cham",
    pages="191--200",
    doi="10.1007/978-3-319-54211-9_15",
}

@article{Abouzaid2022,
title = {Locally positive stabilization of infinite-dimensional linear systems by state feedback},
journal = {European Journal of Control},
volume = {63},
pages = {1-13},
year = {2022},
issn = {0947-3580},
doi = {https://doi.org/10.1016/j.ejcon.2021.07.006},
url = {https://www.sciencedirect.com/science/article/pii/S0947358021000972},
author = {B. Abouzaïd and M.E. Achhab and J.N. Dehaye and A. Hastir and J.J. Winkin}
}

@phdthesis{Wintermayr2019,
  title={Positivity in perturbation theory and infinite-dimensional systems},
  author={J. Wintermayr},
  year={2019},
  school={Universit{\"a}tsbibliothek Wuppertal}
}

@phdthesis{Beauthier2011,
  title={The LQ-Optimal Control Problem for Invariant Linear Systems},
  author={C. Beauthier},
  year={2011},
  school={University of Namur, Belgium}
}

@article{Anita2003,
    title = {Positive stabilization of a parabolic equation by controls localized on a curve},
    journal = {Journal of Mathematical Analysis and Applications},
    volume = {286},
    number = {1},
    pages = {107-115},
    year = {2003},
    issn = {0022-247X},
    doi = {https://doi.org/10.1016/S0022-247X(03)00454-2},
    url = {https://www.sciencedirect.com/science/article/pii/S0022247X03004542},
    author = {S. Aniţa and J.P. Raymond}
}

@book{Vinberg2003,
  title={A course in algebra},
  author={E.B. Vinberg},
  volume={56},
  year={2003},
  publisher={American Mathematical Soc.}
}

\end{document}